\documentclass[reqno]{amsart}

\usepackage{amsmath,amssymb,amsthm,amsfonts,graphicx,color,esint,latexsym,bm,graphics,titletoc,cite}
\usepackage{amssymb,latexsym}
\usepackage{amsmath}
\usepackage{amsthm}
\usepackage{graphicx}
\usepackage{hyperref}
\usepackage{titletoc}
\usepackage{palatino,mathpazo}
\numberwithin{equation}{section}
\numberwithin{equation}{section}

\newtheorem{theorem}{Theorem}[section]
\newtheorem{proposition}[theorem]{Proposition}
\newtheorem{lemma}[theorem]{Lemma}

\newtheorem{remark}[theorem]{Remark}

\theoremstyle{definition}

\def\XXint#1#2#3{{\setbox0=\hbox{$#1{#2#3}{\int}$}
     \vcenter{\hbox{$#2#3$}}\kern-.5\wd0}}

\def\R{{\mathbb R}}

\renewcommand{\d}{\delta }

\newcommand{\ve}{\epsilon}

\begin{document}
\title{Blow up solutions for Sinh-Gordon equation with residual mass}
\author{ Weiwei ~Ao}
\address{ Weiwei ~Ao,~Department of mathematics and statistics, Wuhan university, Wuhan, 430072, P.R. China }
\email{wwao@whu.edu.cn}

\author{ Aleks Jevnikar}
\address{ Aleks Jevnikar,~Department of Mathematics, Computer Science and Physics, University of Udine, Via delle Scienze 206, 33100 Udine, Italy}
\email{aleks.jevnikar@uniud.it}

\author{ Wen Yang}
\address{ Wen ~Yang,~Wuhan Institute of Physics and Mathematics, Chinese Academy of Sciences, P.O. Box 71010, Wuhan 430071, P. R. China}
\email{wyang@wipm.ac.cn}
\begin{abstract}
We are concerned with the Sinh-Gordon equation in bounded domains. We construct blow up solutions with residual mass exhibiting either partial or asymmetric blow up, i.e. where both the positive and negative part of the solution blow up. This is the first result concerning residual mass for the Sinh-Gordon equation showing in particular that the concentration-compactness theory of Brezis-Merle can not be extended to this class of problems.
\end{abstract}

\maketitle

{\footnotesize

\medskip

}

\section{Introduction}\label{sec1}
We are concerned with the following Sinh-Gordon equation 
\begin{equation}
\label{sinh-gordon}
\begin{cases}
\Delta u+\rho^+\dfrac{e^u}{\int_{\Omega }e^udx}-\rho^-\dfrac{e^{-u}}{\int_{\Omega}e^{-u}dx}=0\quad &\mbox{ in }\Omega\\
u=0 &\mbox{ on }\partial\Omega.
\end{cases}
\end{equation}
where $\Omega \subset \R^2$ is smooth and bounded and $\rho^+, \rho^-$ are two positive parameters. The latter problem arises as a mean field equation in the study of the equilibrium turbulence \cite{Montgomery, onsager}. Moreover, it is also related to constant mean curvature surfaces \cite{jwyz,wente}. Observe that for $\rho^-=0$ \eqref{sinh-gordon} reduces to the standard Liouville equation which has been extensively studied in the literature. Therefore, many efforts have been done to study existence \cite{bjmr, jev1, jev2, jev3, jev4} and blow up phenomena \cite{ajy,efp,jevnikar-wei-yang,jevnikar-wei-yang2,jwyz,ohtsuka-suzuki,pistoia-ricchiardi,pistoia-ricchiardi2,ricciardi-takahashi} for this class of problems. 

In the present paper we further explore the blow up phenomenon of \eqref{sinh-gordon}.  Let $u_n$ be a sequence of solutions to (\ref{sinh-gordon}) corresponding to $\rho_n^{\pm}\leq C$. Define the positive and negative blow up set as
\begin{equation*}
S_{\pm}:=\left\{x\in \Omega : \ \exists x_n\to \Omega \ s.t. \ \pm u_n(x_n)-\log \int_{\Omega}e^{\pm u_n}dx + \log \rho_n^{\pm} \to +\infty \ as \ n\to \infty \right\}.
\end{equation*}
It is easy to see that $S_{\pm}$ are finite. Moreover, by \cite{ajy} we have $S_{\pm}\cap \partial\Omega=\emptyset$. For $p\in S_{\pm}$ the local mass is defined by
\begin{equation*} 
m_{\pm} (p)=\lim_{r\to 0}\lim_{n\to \infty}\dfrac{\rho_n^{\pm}\int_{B_r(p)}e^{\pm u_n}dx}{\int_{\Omega} e^{\pm u_n}dx}.
\end{equation*}
By \cite{jevnikar-wei-yang,jwyz} we know that $m_{\pm}(p)$ satisfy a quantization property, i.e. $m_{\pm}(p)\in8\pi\mathbb{N}$. Moreover, in view of the relation
\begin{equation*}
(m_+(p)-m_-(p))^2=8\pi(m_+(p)+m_-(p)),
\end{equation*}
see for example \cite{ohtsuka-suzuki}, the couple $(m_+, m_-)$, up to the order, takes the value in the set
\begin{equation} \label{Sigma}
\Sigma=\Big\{8\pi \left(\frac{k(k-1)}{2}, \frac{k(k+1)}{2}\right), \ k\in \mathbb{N}\setminus \{0\} \Big\}.
\end{equation}
Finally, by standard analysis \cite{ohtsuka-suzuki} one has, for $n\to+\infty$,
$$
	\rho_n^{\pm}\dfrac{e^{\pm u_n}}{\int_{\Omega }e^{\pm u_n}dx} \rightharpoonup \sum_{p\in S_{\pm}} m_{\pm}(p)\d_{p} + r_{\pm},
$$
in the sense of measures, where $r_{\pm}\in L^1(\Omega)$ are residual terms. From the above convergence, $\rho^{\pm}$ will be called global masses of the blow up solutions. Observe that both the local masses and the residual terms affect the global masses. In striking contrast with the concentration-compactness theory of Brezis-Merle \cite{bm}, the latter residuals may not be zero a priori. This fact has important effects in the blow up analysis, variational analysis and Leray-Schauder degree theory of \eqref{sinh-gordon}. One of the goals of the present paper is to provide the first explicit example of blow up solutions exhibiting residual terms, thus confirming that the concentration-compactness theory can not be extended to this class of problems.

\medskip

\subsection{Partial blow up}
We start here with a related problem, that is partial blow up with prescribed global mass. More precisely, we look for blowing up solutions $-u_n$ with $\rho^-_n\to 8\pi k, k\in\mathbb{N}$, such that $u_n$ have prescribed global mass $\rho^+_n=\rho^+\in(0, 8\pi)$.
To this end we introduce
\begin{equation}\label{configuration}
\mathcal{F}_k\Omega:=\biggr\{{\bm\xi}:=(\xi_1,\cdots,\xi_k)\in \Omega^k: \ \xi_i\neq \xi_j \mbox{ for }i \neq j\biggr\}
\end{equation}
and consider the following singular (at $\xi_i\in\Omega$) mean field equation:
\begin{equation}\label{equationofz}
\begin{cases}
\Delta z(x,{\bm\xi})+\rho^+\dfrac{h(x, {\bm\xi})e^{z(x,{\bm\xi})}}{\int_{\Omega}h(x, {\bm\xi})e^{z(x,{\bm\xi})}dx}=0 \quad &\mbox{in}~ \Omega,\\
z(x,{\bm\xi})=0 &\mbox{ on }\partial \Omega
\end{cases}
\end{equation}
where ${\bm\xi}\in \mathcal{F}_k\Omega$ and $h(x,{\bm\xi})=e^{-8\pi \sum_{i=1}^kG(x,\xi_i)}$. Here $G(x,y)$ is the Green function of the Laplacian operator in $\Omega$ with Dirichlet boundary condition and we denote its regular part by $H(x,y)$. \eqref{equationofz} is the Euler-Lagrange equation of the functional
$$
	I_{\bm\xi}(z):= \frac12 \int_{\Omega}|\nabla z|^2dx-\rho^+\log\left( \int_{\Omega} h(x,\bm\xi)e^z dx \right).
$$ 
To the latter functional and (a combination of) the Green functions we associate the following map:
\begin{equation}
\label{Lambda}
\Lambda({\bm\xi}):=\frac{1}{2}I_{\bm\xi}(z(\cdot,{\bm\xi}))-32\pi^2\Big( \sum_{i=1}^kH(\xi_i,\xi_i) +\sum_{j\neq i}G(\xi_i,\xi_j)\Big).
\end{equation}
It is known by \cite{bartolucci-lin} that if $\Omega$ is simply connected and $\rho^+\in (0, 8\pi)$, then for any ${\bm\xi}\in \mathcal{F}_k\Omega$ there exists a unique solution to (\ref{equationofz}) and the solution is non-degenerate, in the sense that the linearized problem admits only the trivial solution. Then, by making use of the implicit function theorem it is not difficult to show that the function $\Lambda$ is smooth, see for example \cite{dpr}. Finally, as in \cite{li}, a compact set $\mathcal{K}\subset \mathcal{F}_k\Omega$ of critical points of $\Lambda$ is said to be $C^1$-stable if, fixed a neighborhood $\mathcal{U}$ of $\mathcal{K}$, any map $\Phi: \mathcal{U}\to \R$ sufficiently close to $\Lambda$ in $C^1$-sense has a critical point in $\mathcal{U}$.

The first result of this paper is the following.
\begin{theorem}\label{thm}
Let $\Omega$ be simply connected, $\rho^+\in (0, 8\pi)$ and let $\mathcal{K}\subset \mathcal{F}_k\Omega$, $k\in \mathbb{N}$, be a $C^1$-stable set of critical points of $\Lambda$. Then, there exists $\lambda_0>0$ such that for any $\lambda \in (0,\lambda_0)$ there exists $u_\lambda$ solution of \eqref{sinh-gordon} with $\rho_\lambda^{\pm}$ such that, for $\lambda\to0$
\begin{itemize}
\item[1.]$
\rho^+_\lambda =\rho^+ , \ \rho^-_\lambda \to 8k\pi.$
\item[2.]
There exist ${\bm\xi}(\lambda)\in \mathcal{F}_k\Omega$ and $\delta_i(\lambda)>0$ such that $d({\bm\xi}, \mathcal{K})\to 0, \ \delta_i\to 0$ and
\begin{equation*}
u_\lambda(x)\to z(x,{\bm\xi})-\sum_{i=1}^k\Big(\log \frac{1}{(\delta_i^2+|x-\xi_i|^2)^2}+8\pi H(x, \xi_i)  \Big) \quad \mbox{in } H^1_0(\Omega),
\end{equation*}
where $z$ solves (\ref{equationofz}).
\end{itemize}
\end{theorem}

Some comments are in order. The assumptions $\Omega$ simply connected  and $\rho^+\in (0, 8\pi)$ guarantee the existence of a unique non-degenerate solution to \eqref{equationofz}: in general, the above result holds true whenever such solution exists. For example, one can drop the condition on $\Omega$ by assuming $\rho^+$ to be sufficiently small, see for example \cite{dpr}. 

On the other hand, if $\Omega$ is simply connected and $\rho^+\in (0, 8\pi)$ it is not difficult to show that for $k=1$ the minimum of $\Lambda$ is a $C^1$-stable set of critical points of $\Lambda$, see for example \cite{dpr}. Moreover, for non-simply connected domains the function $\Lambda$ always admits a $C^1$-stable set of critical points \cite{dpkm}. 

Therefore, the conclusion of Theorem \ref{thm} holds true if either $\Omega$ is simply connected, $\rho^+\in (0, 8\pi)$ and $k=1$, or $\Omega$ is multiply connected, $\rho^+$ sufficiently small and $k\geq1$. Finally, the location of the blow up set can be determined by using the following expression, which can be derived similarly as in \cite{dpr}:
\begin{equation} \label{der}
	\partial_{\xi_j}\Lambda({\bm\xi})=8\pi\frac{\partial z}{\partial x}(\xi_j,{\bm\xi})-32\pi^2\Big( \frac{\partial H}{\partial x}(\xi_j,\xi_j)+\sum_{i\neq j}\frac{\partial G}{\partial x}(\xi_i,\xi_j)\Big).
\end{equation}

\medskip

\subsection{Asymmetric blow up}
We next construct blow up solutions with residual mass exhibiting asymmetric blow up, i.e. where both the positive and negative part of the solution blow up. Since the local masses $(m_+,m_-)$ belong to the set $\Sigma$ defined in \eqref{Sigma}, for $k \geq 2$ we look for blowing up solution $u_n$ with $\rho_n^-\to 4\pi k(k+1)$ and $\rho_n^+= \rho^+=4\pi k(k-1)+\rho_0$, where $\rho_0\in (0, 8\pi)$ is a fixed residual mass. For simplicity of presentation we assume that $k$ is odd, the case of $k$ even being similar. We consider here $l$-symmetric domains $\Omega$ with $l\geq2$ even, i.e. if $x\in\Omega$ then $\mathcal{R}_l\cdot x\in\Omega$, where
\begin{equation}\label{symmetry}
\mathcal{R}_l:=\left(\begin{array}{cc}
\cos\frac{2\pi}{l}&\sin\frac{2\pi}{l}\vspace{0.2cm}\\
-\sin\frac{2\pi}{l}&\cos\frac{2\pi}{l}
\end{array}
\right), \quad \mbox{$l\geq 2$ even}.
\end{equation}
Consider then the following singular (at $x=0$) mean field equation:
\begin{equation}
\label{equationofz-1}
\begin{cases}
\Delta z(x)+\rho_0\dfrac{e^{z(x)-8k\pi G(x,0)}}{\int_{\Omega}e^{z(x)-8k\pi G(x,0)}dx}=0 \quad &\mbox{in} ~\Omega,\\
z(x)=0 &\mbox{ on }\partial \Omega.
\end{cases}
\end{equation}
Again by \cite{bartolucci-lin} we know that if $\Omega$ is simply connected and $\rho^+\in (0, 8\pi)$, then there exists a unique non-degenerate solution to (\ref{equationofz-1}).

The second result of this paper is the following.
\begin{theorem}\label{thm-1}
Let $\Omega$ be a simply connected $l$-symmetric domain according to \eqref{symmetry} and $\rho^+=4\pi k(k-1)+\rho_0$ with $k\in\mathbb{N}$ odd and $\rho_0\in (0, 8\pi)$. Then, there exists $\lambda_0>0$ such that for any $\lambda \in (0,\lambda_0)$, there exists $u_\lambda$ solution of \eqref{sinh-gordon} with $\rho_\lambda^{\pm}$ such that, for $\lambda \to 0$
\begin{itemize}
\item[1.]$
\rho^+_\lambda =\rho^+ , \ \rho^-_\lambda \to 4\pi k(k+1).$
\item[2.]
There exists $\delta_i(\lambda)\to 0$ such that
\begin{equation*}
u_\lambda(x)\to z(x)+\sum_{i=1}^k(-1)^i\Big(\log \frac{1}{(\delta_i^{\alpha_i}
+|x|^{\alpha_i})^2}+4\pi \alpha_iH(x, 0)  \Big)  \quad \mbox{in } H^1_0(\Omega), \quad \alpha_i=4i-2,
\end{equation*}
where $z$ solves (\ref{equationofz-1}).
\end{itemize}
\end{theorem}

Observe that the assumption $\Omega$ simply connected and $\rho_0\in (0, 8\pi)$ is used only to ensure the existence of a non-degenerate solution to \eqref{equationofz-1}: in general, the above result holds true whenever such solution exists. On the other hand, the symmetry condition of the domain is imposed to rule out the degeneracy of the singular Liouville equation. 

\medskip

The argument follows the strategy introduced in \cite{dpr,dpr1} for the Toda system, that is a system of Liouville-type equations, and it is based on perturbation method starting from an approximate solution and studying the invertibility of the linearized problem. The main difficulty is due to the coupling of the local and global nature of the problem since we are prescribing both the local and global masses. In particular, blow up solutions of \eqref{sinh-gordon} with local masses $(4\pi k(k-1), 4\pi k(k+1))$ have been constructed in \cite{grossi-pistoia} by superposing $k$ different bubbles with alternating sign. Gluing the solution of \eqref{equationofz-1} to the latter blow up solutions we are able to construct blow up solutions with residual mass, that is with $\rho^+_n=\rho^+=4\pi k(k-1)+\rho_0$ and $\rho^-_n\to 4\pi k(k+1)$ for any $k\geq 2$. In this generality the latter construction is quite delicate and technically more difficult compared to the one in \cite{dpr1, grossi-pistoia}. We remark that the same strategy can be carried out for more general asymmetric Sinh-Gordon equations, for example for the Tzitz\'eica equation \cite{jevnikar-yang}.

The paper is organized as follows. Section \ref{sec2} contains some notation and preliminary results which will be used in the paper. Section \ref{sec3} is devoted to the proof of Theorem \ref{thm} while the proof of Theorem \ref{thm-1} is derived in Section \ref{sec4}.


\vspace{0.5cm}
\section{Preliminaries}\label{sec2}

In this section we collect some notation and useful information that we will use in this paper. We shall write
$$\|u\|=\Big( \int_{\Omega}|\nabla u|^2dx\Big)^{\frac{1}{2}} \quad \mbox{and} \quad \|u\|_p=\Big( \int_{\Omega}u^pdx\Big)^{\frac{1}{p}} $$
to denote the norm in $H_0^1(\Omega)$ and in $L^p(\Omega)$, respectively, for $1\leq p\leq +\infty$.
For $\alpha\geq 2$, let us define the Hilbert spaces:
\begin{equation*}
L_{\alpha}(\R^2):=L^2\Big( \R^2, \dfrac{|y|^{\alpha-2}}{(1+|y|^\alpha)^2}dy \Big),
\end{equation*}
\begin{equation*}
H_{\alpha}(\R^2):=\{u\in W_{loc}^{1,2}(\R^2)\cap L_\alpha(\R^2): \|\nabla u\|_{L^2(\R^2)}<\infty\},
\end{equation*}
with $\|u\|_{L_{\alpha}}$ and $\|u\|_{H_{\alpha}}:=(\|\nabla u\|_{L^2(\R^2)}^2+\|u\|_{L_{\alpha}}^2)^{\frac{1}{2}}$ denoting their norms, respectively. For simplicity, we will denote $L_2$ and $H_2$ by $L$ and $H$, respectively. Let us recall that the embedding $H_\alpha(\R^2)\to L_\alpha(\R^2)$ is compact \cite{dpr1}. Moreover, for $v\in L^{p}(\Omega)$ let $u$ be the solution of
\begin{equation*}
\Delta u=v \mbox{ in }\Omega, \qquad u=0\mbox{ on }\partial \Omega.
\end{equation*}
Then one has $\|u\|\leq c_p \|v\|_p$ for some constant $c_p>0$ depending only on $\Omega$ and $p>1$.

\medskip

The symbol $B_r(p)$ will stand for the open metric ball of
radius $r$ and center $p$. To simplify the notation we will write $B_r$ for balls which are centered at $0$.
Throughout the whole paper $c,C$ will stand for constants which
are allowed to vary among different formulas or even within the same line.

\vspace{0.5cm}
\section{Partial blow up}\label{sec3}

\subsection{ Approximate solutions}\label{sec3.1}

In order to prove Theorem \ref{thm} we introduce the associated equation
\begin{equation}
\label{mainproblem}
\begin{cases}
\Delta u+\rho^+\dfrac{e^u}{\int_\Omega e^u}-\lambda e^{-u}=0 \quad \mbox{ in }\Omega, \\
u=0 \quad \mbox{ on }\partial \Omega
\end{cases}
\end{equation}
where $\lambda>0 $ will be suitably chosen small. First let us introduce the approximate solutions we will use. Recall that solutions of the following regular Liouville equation:
\begin{equation*}
\Delta w+e^w=0 \quad \mbox{ in }\R^2, \qquad \int_{\R^2}e^w dx<\infty,
\end{equation*}
are given by
\begin{equation*}
w_{\delta, \xi}(x)=\log \frac{8\delta^2}{(\delta^2+|x-\xi|^2)^2}
\end{equation*}
for $\delta>0, \ \xi\in \R^2$ and we set
\begin{equation*}
w(x)=\log \frac{8}{(1+|x|^2)^2}.
\end{equation*}
Since we are considering Dirichlet boundary condition, let us introduce the projection:
\begin{equation*}
\Delta Pu=\Delta u \quad \mbox{ in }\Omega, \qquad Pu=0 \quad \mbox{ on }\partial \Omega.
\end{equation*}
It is well-known that
\begin{equation}\label{bubble-projection}
Pw_{\delta, \xi}(x)=w_{\delta, \xi}(x)-\log 8\delta^2+8\pi H(x, \xi)+O(\delta^2) \quad \mbox{in $C^1$-sense},
\end{equation}
where $H(x,y) $ is the regular part of the Green's function of the Dirichlet Laplacian in $\Omega$, $G(x,y)=\frac{1}{2\pi}\log \frac{1}{|x-y|}+H(x,y)$.

Let $k\geq 1$, fix ${\bm\xi}\in \mathcal{F}_k\Omega$ and consider $z(x, {\bm\xi})$ the unique solution to (\ref{equationofz}). The approximate solutions we will use are given by
\begin{equation}\label{approximation}
W=z(x, {\bm\xi})-\sum_{i=1}^k Pw_i(x), \quad w_i(x)=w_{\delta_i,\xi_i}(x),
\end{equation}
where the parameters $\delta_i$ are suitably chosen such that 
\begin{equation}\label{delta}
8\delta_i^2=\lambda d_i({\bm\xi}), \quad d_i({\bm\xi})=\exp\Big[8\pi (H(\xi_i,\xi_i)+\sum_{j\neq i} G(\xi_i,\xi_j))-z(\xi_i,{\bm\xi})  \Big].
\end{equation}
Our aim is to find a solution $u$  to (\ref{mainproblem}) of the form $u=W+\phi$ where $\phi $ is small in some sense. Before we go further, let us first collect some useful well-known facts. 

Any solution $\psi\in H$ of
\begin{equation*}
\Delta \psi+e^{w_{\delta, \xi}}\psi=0\quad \mathrm{ in }\quad \R^2,
\end{equation*}
can be expressed as a linear combination of
\begin{equation*}
Z_{\delta, \xi}^0(x)=\frac{\delta^2-|x-\xi|^2}{\delta^2+|x-\xi|^2}, \quad Z_{\delta,\xi}^i(x)=\frac{x_i-\xi_i}{\delta^2+|x-\xi|^2}, \ i=1,2.
\end{equation*}
Moreover, the projections of $Z_{\delta, \xi}^i$ have the following expansion:
\begin{equation}\label{pz}
PZ_{\delta,\xi}^0(x)=Z_{\delta, \xi}^0(x)+1+O(\delta^2),\quad PZ_{\delta, \xi}^i(x)=Z_{\delta, \xi}^i(x)+O(1), \ i=1,2 \quad \mbox{in $C^1$-sense}.
\end{equation}
Finally, by straightforward computations and taking into account the choice of $\lambda$ in \eqref{delta} the following estimates hold true.
\begin{lemma}
\label{le3.3}
For any $\mathcal{C}\subset \mathcal{F}_k\Omega$ compact and ${\bm\xi}\in \mathcal{C}$, one has
\begin{equation*}
\|Pw_i\|=O(|\log \lambda|^{\frac{1}{2}}),\quad \|\nabla_{\bm\xi}Pw_i\|=O(\lambda^{-\frac{1}{2}}),
\end{equation*}
\begin{equation*}
\|W\|=O(|\log \lambda|^{\frac{1}{2}}),\quad \|\nabla_{\bm\xi}W\|=O(\lambda^{-\frac{1}{2}}),
\end{equation*}
and there exists some $a>0$ such that for any $i=1,\cdots,k$ and $j=1,2$, it holds that
\begin{equation}\label{estimateofkernel}
\|PZ_i^j\|=a\lambda^{-\frac{1}{2}}(1+o(1)) ,\quad \|\nabla_{\bm\xi}PZ_i^j\|=O\left(\frac{1}{\lambda}\right),
\end{equation}
and
\begin{equation}\label{estimateofkernel1}
\langle PZ_i^j, PZ_l^k\rangle =o\left(\frac{1}{\lambda}\right)\quad \mbox{ if }\quad i\neq l \mbox{ or }j\neq k.
\end{equation}
\end{lemma}

\smallskip


\subsection{Estimate of the error}\label{sec3.2}
We next estimate the error of the approximate solution:
\begin{equation*}
R=\Delta W+\rho^+ \frac{e^W}{\int_\Omega e^W}-\lambda e^{-W}.
\end{equation*}

\begin{lemma}\label{estimateoferror}
For any $p\geq 1$ we have, for ${\bm\xi}\in \mathcal{C}\subset \mathcal{F}_k\Omega$, $\mathcal{C}$ compact, 
$$
\|R\|_p=O(\lambda^{\frac{2-p}{2p}}), \quad\|\partial_{\bm\xi}R\|_p=O(\lambda^{\frac{1-p}{p}}).
$$
\end{lemma}

\begin{proof}
By the definition of $W$,
\begin{equation*}
\begin{split}
R&=\Delta W+\rho^+ \frac{e^W}{\int_\Omega e^W}-\lambda e^{-W}\\
&=\Delta(z(x,{\bm\xi})-\sum_i Pw_i)+\rho^+\frac{e^{z(x,{\bm\xi})-\sum_i Pw_i}}{\int_\Omega e^{z(x,{\bm\xi})-\sum_i Pw_i}}-\lambda e^{\sum_i Pw_i-z(x,{\bm\xi})}\\
&=\Big(\sum_i e^{w_i}-\lambda e^{\sum_i Pw_i-z(x,{\bm\xi})}\Big)+\Big(\Delta z(x,{\bm\xi})+\rho^+\frac{e^{z(x,{\bm\xi})-\sum_i Pw_i}}{\int_\Omega e^{z(x,{\bm\xi})-\sum_i Pw_i}}\Big)\\
&:=E_1(x)+E_2(x).
\end{split}
\end{equation*}

\smallskip

\noindent {\bf Estimate of $E_1=\Big(\sum_i e^{w_i}-\lambda e^{\sum_i Pw_i-z(x,{\bm\xi})}\Big)$.} Take $\eta>0$ such that $|\xi_i-\xi_j|\geq2\eta$ and $d(\xi_i,\partial\Omega)\geq2\eta$.  First, using (\ref{bubble-projection}), we have
\begin{equation*}
\begin{split}
W&=z(x,{\bm\xi})-\sum_i Pw_i=z(x,{\bm\xi})-\sum_i \Big[ \log \frac{1}{(\delta_i^2+|x-\xi_i|^2)^2}+8\pi H(x, \xi_i)\Big]+O(\lambda).
\end{split}
\end{equation*}
Hence, on $B_{\eta}(\xi_i)$, writing $x=\xi_i+\delta_i y$, one has
\begin{equation*}
\begin{split}
e^{-W(x)}&=e^{\sum_{i=1}^k \Big[ \log \frac{1}{(\delta_i^2+|x-\xi_i|^2)^2}+8\pi H(x, \xi_i)\Big]-z(x, {\bm\xi})}(1+O(\lambda))\\
&=e^{w(y)}\cdot \exp\Big(8\pi H(\xi_i,\xi_i)+\sum_{j\neq i}8\pi G(\xi_i,\xi_j)-4\log \delta_i-\log 8-z(\xi_i,{\bm\xi})  \Big)(1+O(\lambda)+O(\delta_i|y|))\\
&=\frac{d_i({\bm\xi})}{8\delta_i^4}e^{w(y)}
(1+O(\lambda)+O(\delta_i|y|)).
\end{split}
\end{equation*}
Thus
\begin{equation}\label{eq1}
\begin{split}
e^{w_i}-\lambda e^{-W(x)}&=\frac{8}{\delta_i^2(1+|y|^2)^2}\left[1-\frac{\lambda}{8\delta_i^2}d_i({\bm\xi})+O(\lambda)+O(\delta_i|y|)\right]\\
&=O\left(\frac{1}{(1+|y|^2)^2}\right)+O\left(\frac{|y|}{\lambda^{\frac{1}{2}}(1+|y|^2)^2}\right).
\end{split}
\end{equation}
It follows that
\begin{equation*}
\|e^{w_i}-\lambda e^{-W(x)}\|_{L^p(B(\xi_i,\eta))}=O(\lambda^{\frac{2-p}{2p}})~\mbox{ for \ any }p\geq 1.
\end{equation*}
Moreover,
\begin{equation*}
\|e^{w_j}\|_{L^\infty(B(\xi_i,\eta))}=O(\lambda) \mbox{ for } j\neq i \ \mbox{ and } \ \|e^{w_i}-\lambda e^{-W(x)}\|_{L^\infty(\Omega\setminus \cup_i B(\xi_i,\eta))}=O(\lambda).
\end{equation*}
Combining the above estimates,
\begin{equation}\label{estimateofe1}
\|E_1\|_{p}=O(\lambda^{\frac{2-p}{2p}}) \mbox{ for }p\geq 1.
\end{equation}

\smallskip

\noindent {\bf Estimate of $E_2=\Big(\Delta z(x,{\bm\xi})+\rho^+\frac{e^{z(x,{\bm\xi})-\sum_i Pw_i}}{\int_\Omega e^{z(x,{\bm\xi})-\sum_i Pw_i}}\Big)$.} First of all,
\begin{equation}\label{estimateofw}
\begin{split}
W&=z(x,{\bm\xi})-\sum_i Pw_i\\
&=z(x,{\bm\xi})+2\sum_i \log (\delta_i^2+|x-\xi_i|^2)-8\pi\sum_i H(x,\xi_i)+O(\lambda)\\
&=\log h(x, {\bm\xi})+z(x,{\bm\xi})+2\sum_i\log\frac{\delta_i^2+|x-\xi_i|^2}{|x-\xi_i|^2}+O(\lambda),
\end{split}
\end{equation}
where
\begin{equation*}
h(x, {\bm\xi})=\prod_{i=1}^k |x-\xi_i|^4\exp[-8\pi H(x, \xi_i)]
=\prod_{i=1}^k\exp\Big(-8\pi G(x,\xi_i)\Big).
\end{equation*}
So
\begin{equation}\label{eq2}
e^W=h(x,{\bm\xi})e^{z(x,{\bm\xi})}+O(\lambda).
\end{equation}
One has
\begin{equation*}
E_2=\Delta z(x,{\bm\xi})+\rho^+\frac{e^W}{\int_{\Omega} e^W}=\Delta z(x, {\bm\xi})+\rho^+\frac{h(x,{\bm\xi})e^{z(x,{\bm\xi})}}{\int_{\Omega}h(x,{\bm\xi})e^{z(x,{\bm\xi})}}+O(\lambda)
=O(\lambda),
\end{equation*}
since $z(x,{\bm\xi})$ is a solution of (\ref{equationofz}). Thus
\begin{equation}\label{estimateofe2}
\|E_2\|_{\infty}=O(\lambda).
\end{equation}

\smallskip

\noindent{\bf Derivative of $E_1$. }Next we consider the derivatives. By straightforward computations we get
\begin{equation*}
\begin{split}
\partial_{\xi_i^j}E_1=~&\sum_\ell e^{w_\ell}\partial_{\xi_i^j}w_{\ell}+\lambda e^{-W}\partial_{\xi_i^j}W\\
=~&\lambda e^{-W}\partial_{\xi_i^j} z(x, {\bm\xi})+(\sum_i e^{w_i}-\lambda e^{-W})\sum_{\ell=1}^k\partial_{\xi_i^j}Pw_\ell -\sum_{\ell}e^{w_\ell}\partial_{\xi_i^j}(Pw_\ell-w_\ell)-\sum_{\ell\neq i}e^{w^i}\partial_{\xi_i^j}Pw_\ell\\
:=~&I_1+I_2+I_3+I_4.
\end{split}
\end{equation*}
It is then not difficult to show that
\begin{equation*}
\begin{split}
\|I_1\|_p&\leq \|E_1\|_p+\sum_i\|e^{w_i}\|_p=O(\lambda^{\frac{1-p}{p}}),\\
\|I_2\|_p&\leq \|E_1\|_p\|\partial_{\bm\xi}Pw_j\|_\infty=O(\lambda^{\frac{1-p}{p}}),\\
\|I_3\|_p&\leq \|e^{w_i}\|_p \|\partial_{\bm\xi}(Pw_j-w_j)\|_\infty=O(\lambda^{\frac{1-p}{p}}),\\
\|I_4\|_p&=0.
\end{split}
\end{equation*}
Combining all the above estimates,
\begin{equation}\label{derivativeofe1}
\|\partial_{\bm\xi}E_1\|_p=O(\lambda^{\frac{1-p}{p}}).
\end{equation}

\medskip

\noindent{\bf Derivative of $E_2$. }The estimate of the derivative of $E_2$ is analogous. Using the equation satisfied by $z(x,{\bm\xi})$ in (\ref{equationofz}) and (\ref{estimateofw}),
\begin{equation*}
\begin{split}
\frac{1}{\rho^+}\partial_{\xi_i^j}E_2=~&-\frac{(\partial_{\xi_i^j}z(x, {\bm\xi})h+\partial_{\xi_i^j}h)e^{z(x,{\bm\xi})}}{\int_{\Omega}he^{z(x, {\bm\xi})}dx}+\frac{he^{z(x, {\bm\xi})}\int_{\Omega}(\partial_{\xi_i^j}z(x, {\bm\xi})h+\partial_{\xi_i^j}h)e^{z(x,{\bm\xi})}}{(\int_{\Omega}he^{z(x, {\bm\xi})}dx)^2}\\
&+\frac{e^W\partial_{\xi_i^j}W}{\int_{\Omega}e^W}-\frac{e^W\int_{\Omega}e^W\partial_{\xi_i^j}Wdx}{(\int_\Omega e^W)^2}\\
=~&O(\lambda).
\end{split}
\end{equation*}
Thus we have
\begin{equation}\label{derivativeofe2}
\|\partial_{\bm\xi}E_2\|_\infty=O(\lambda).
\end{equation}
Finally, combining the estimates for $E_1$ and $E_2$, we have
\begin{equation*}
\|R\|_p=O(\lambda^{\frac{2-p}{2p}}), \quad \|\partial_{\bm\xi}R\|_p=O(\lambda^{\frac{1-p}{p}}).
\end{equation*}
\end{proof}

\subsection{The linear operator}\label{sec3.3}
In this subsection, we consider the following problem:  given $h\in H_0^1(\Omega)$ we look for a function $\phi\in H_0^1(\Omega)$ and $c_{ij}$ such that
\begin{equation}\label{linearproblem}
\begin{cases}
\Delta \phi+\rho^+\left(\dfrac{e^W\phi}{\int_\Omega e^Wdx}-\dfrac{e^W\int_\Omega e^W\phi dx}{(\int_\Omega e^Wdx)^2}\right)+\sum_{i=1}^k e^{w_i}\phi=\Delta h+\sum_{i,j}c_{ij}e^{w_i}Z_i^j,\\
\\
\int_{\Omega}\nabla \phi \nabla PZ_i^j dx=0,~j=1,2,~i=1,\cdots,k.
\end{cases}
\end{equation}

First we have the following apriori estimate:
\begin{lemma}\label{aprioriestimate}
Let $\mathcal{C}\subset \mathcal{F}_k\Omega$ be a fixed compact set. Then, there exist $\lambda_0>0$ and $C>0$ such that for any $\lambda\in (0,\lambda_0)$, ${\bm\xi}\in \mathcal{C}$ and $h\in H_0^1(\Omega)$, any solution $\phi\in H_0^1(\Omega)$ of
\begin{equation}\label{linearproblem1}
\begin{cases}
\Delta \phi+\rho^+\left(\dfrac{e^W\phi}{\int_\Omega e^Wdx}-\dfrac{e^W\int_\Omega e^W\phi dx}{(\int_\Omega e^Wdx)^2} \right)+\sum_{i=1}^k e^{w_i}\phi=\Delta h,\\
\\
\int_{\Omega}\nabla \phi \nabla PZ_i^j dx=0,~j=1,2,~i=1,\cdots,k,
\end{cases}
\end{equation}
satisfies
\begin{equation*}
\|\phi\|\leq C|\log \lambda| \|h\|.
\end{equation*}
\end{lemma}
\begin{proof}
We prove it by contradiction. Assume there exist $\lambda_n\to 0$, ${\bm\xi}_n\to {\bm\xi}^*\in \mathcal{F}_k \Omega$, $h_n\in H_0^1(\Omega)$ and $\phi_n \in H_0^1(\Omega)$ which solves (\ref{linearproblem1}) with
\begin{equation*}
\|\phi_n\|=1, \ |\log \lambda_n|\|h_n\|\to 0 \mbox{ as }n\to \infty.
\end{equation*}
For $i=1, \cdots, k$, define $\tilde{\phi}_i(y)$ as
\begin{equation*}
\begin{split}
\tilde{\phi}_i(y)=
\begin{cases}
\phi_i(\delta_i y+\xi_i), \quad &y\in \tilde{\Omega}_i=\frac{\Omega-\xi_i}{\delta_i},\\
0, &y\in \R^2\setminus \tilde{\Omega}_i.
\end{cases}
\end{split}
\end{equation*}

\noindent {\bf Step 1. } We claim that
\begin{equation}\label{limit1}
\tilde{\phi}_i(y)\to \gamma_i\frac{1-|y|^2}{1+|y|^2} \mbox{ weakly \ in\  $H(\R^2)$ \ and \ strongly \ in \ $L(\R^2)$},
\end{equation}
and
\begin{equation}\label{limit2}
\phi\to 0 \mbox{ weakly \ in \ $H_0^1(\Omega)$\ and \ strongly \ in \ $L^q(\Omega)$ \ for $q\geq 2$}.
\end{equation}
Let $\psi\in C_0^\infty(\Omega\setminus \{\xi^*_1, \cdots,\xi_k^*\})$, multiply equation (\ref{linearproblem1}) by $\psi$ and integrate, then
\begin{equation*}
\begin{aligned}
&-\int_{\Omega}\nabla \psi \nabla \phi+\sum_{i=1}^k\int_{\Omega}e^{w_i}\phi\psi dx
+\rho^+\left(\dfrac{\int_{\Omega}e^W\phi\psi dx}{\int_{\Omega}e^Wdx}-\dfrac{\int_{\Omega}e^W\phi dx
\int_{\Omega} e^W\psi dx}{(\int_{\Omega}e^Wdx)^2}\right)\\
&=\int_{\Omega }\Delta h \psi dx.
\end{aligned}
\end{equation*}
By the assumption on $\phi$, using the fact that in $\Omega\setminus\{\xi_1^*,\cdots, \xi_k^*\}, \ e^{w_i}=O(\lambda)$ and $e^W=h(x, {\bm\xi})e^{z(x,{\bm\xi})}+O(\lambda)$, one has
\begin{equation*}
\phi \to \phi^* \mbox{ weakly \ in \ $H_0^1(\Omega)$ \ and \ strongly \ in \ $L^q(\Omega)$ \ for \ $q\geq 2$},
\end{equation*}
which gives
\begin{equation*}
-\int_{\Omega}\nabla \phi^* \nabla \psi dx+\rho^+\left( \dfrac{\int_{\Omega}he^z\phi^*\psi dx}{\int_{\Omega} he^zdx}-\dfrac{\int_{\Omega}he^z\psi dx\int_{\Omega}he^z\phi^*dx}{(\int_{\Omega}he^zdx)^2}\right)=0.
\end{equation*}
So $\|\phi^*\|_{H_0^1(\Omega)}\leq 1$ and it solves
\begin{equation*}
\Delta \phi^*+\rho^+\left(\dfrac{he^z\phi^*}{\int_{\Omega}he^zdx}
-\dfrac{he^z\int_{\Omega}he^z\phi^*dx}{(\int_{\Omega} he^zdx)^2}\right)=0.
\end{equation*}
By the non-degeneracy of $z(x, {\bm\xi})$, we can get that $\phi^*=0$.  Thus (\ref{limit2}) is proved.

Now let us prove (\ref{limit1}). Multiplying (\ref{linearproblem1}) again by $\phi$ and integrating,
\begin{equation*}
\int_{\Omega}|\nabla \phi|^2dx-\sum_{i=1}^k\int_{\Omega}e^{w_i}\phi^2dx-\rho^
+\left(\dfrac{\int_{\Omega}e^W\phi^2dx}{\int_{\Omega}e^Wdx}
-\dfrac{(\int_{\Omega}e^W\phi dx)^2}{(\int_{\Omega}e^Wdx)^2} \right)
=\int_{\Omega}\nabla h\nabla \phi dx.
\end{equation*}
From the above equation, one can get that
\begin{equation*}
\begin{split}
\int_{\tilde{\Omega}_i}e^w \tilde{\phi}_i^2dx&=\int_{\Omega}e^{w_i}\phi^2dx\\
&\leq \int_{\Omega}|\nabla\phi|^2dx-\rho^+\left(\dfrac{\int_{\Omega}e^W\phi^2dx}{\int_{\Omega}e^Wdx}
-\dfrac{(\int_{\Omega}e^W\phi dx)^2}{(\int_{\Omega}e^Wdx)^2}\right)-\int_{\Omega}\nabla h\nabla \phi dx\\
&\leq 1+o(1)+\|h\|=O(1)
\end{split}
\end{equation*}
where we used (\ref{limit2}).  So we get that $\tilde{\phi}_i$ is bounded in $H(\R^2)$. There exists $\tilde{\phi}_0$ such that
\begin{equation*}
\tilde{\phi}_i \to \tilde{\phi}_0 \mbox{ weakly \ in \ $H(\R^2)$ \ and \ strongly \ in \ $L(\R^2)$}.
\end{equation*}
Let $\tilde{\psi}\in C_0^\infty(\R^2)$ and define $\psi_i=\tilde{\psi}(\frac{x-\xi_i}{\delta_i})$. Multiplying (\ref{linearproblem1}) by $\psi_i$ and integrating over $\Omega$,
\begin{equation}\label{psi_i}
\begin{aligned}
&\int_{\Omega} \nabla \phi\nabla \psi_idx
-\sum_{j}\int_{\Omega}e^{w_j}\phi\psi_idx
-\rho^+\left(\dfrac{\int_{\Omega}e^W\phi\psi_idx}{\int_{\Omega}e^Wdx}-\dfrac{\int_{\Omega} e^W\phi dx
\int_{\Omega} e^W\psi_idx}{(\int_{\Omega} e^Wdx)^2} \right)\\
&=\int_{\Omega} \nabla h\nabla \psi_idx.
\end{aligned}
\end{equation}
Since $\psi_i(x)=0$ if $|x-\xi_i|\geq R\delta_i$ for some $R>0$, we have
\begin{equation*}
\int_{\Omega}e^{w_j}\phi\psi_idx=O(\delta_j^2) \quad \mbox{ for }j\neq i.
\end{equation*}
Passing to the limit in (\ref{psi_i}), we have
\begin{equation*}
\int_{\R^2}\nabla \tilde{\phi}_0\nabla \tilde{\psi}dx-\int_{\R^2}e^w \tilde{\phi}_0\tilde{\psi}dx=0.
\end{equation*}
Moreover, by the orthogonality condition in (\ref{linearproblem1}), we have
\begin{equation*}
\int_{\R^2}\tilde{\phi}_0e^w \frac{y_j}{1+|y|^2}dy=0, \ j=1,2.
\end{equation*}
So we deduce that
\begin{equation*}
\tilde{\phi}_0=\gamma_i\frac{1-|y|^2}{1+|y|^2}.
\end{equation*}

\noindent{\bf Step 2. } We claim that $\gamma_i=0$ for $i=1,\cdots,k$. Multiplying equation (\ref{linearproblem1}) by $PZ_i^0$ and integrate over $\Omega$,
\begin{equation}\label{pz1}
\begin{aligned}
&\int_{\Omega}\nabla \phi\nabla PZ_i^0dx-\sum_{j}\int_{\Omega}e^{w_j}\phi PZ_i^0dx
-\rho^+\left(\dfrac{\int_{\Omega}e^W\phi PZ_i^0dx}{\int_{\Omega}e^Wdx}
-\dfrac{\int_{\Omega}e^W\phi dx\int_{\Omega}e^W PZ_i^0dx}{(\int e^Wdx)^2} \right)\\
&=\int_{\Omega}\nabla h\nabla PZ_i^0dx.
\end{aligned}
\end{equation}
Since
\begin{equation*}
\int_{\Omega}\nabla \phi\nabla PZ_i^0dx=\int_{\Omega}e^{w_i}\phi Z_i^0dx=\int_{\tilde{\Omega}_i}e^w Z^0 \tilde{\phi}_idy
\end{equation*}
where $Z^0=\frac{1-|y|^2}{1+|y|^2}$ and by \eqref{pz},
\begin{equation*}
\begin{split}
\sum_j \int_{\Omega}e^{w_j}\phi PZ_i^0dx&=\int_{\Omega}e^{w_i}\phi PZ_i^0dx
+\sum_{j\neq i}\int_{\Omega}e^{w_j}\phi PZ_i^0dx\\
&=\int_{\tilde{\Omega}_i}e^w\tilde{\phi}_i (1+Z^0(y)+O(\delta_i^2))dy
+\sum_{j\neq i}\int_{\Omega}e^{w_j}\phi PZ_i^0dx\\
&=\int_{\tilde{\Omega}_i}e^w \tilde{\phi}_i (1+Z^0(y))dy+O(\lambda^{\frac{1}{p}}),
\end{split}
\end{equation*}
for some $p>1$, by H\"{o}lder inequality. Moreover, by (\ref{limit2}), \eqref{pz} and (\ref{estimateofw}), one has
\begin{equation*}
\rho^+\left(\dfrac{\int_{\Omega}e^W\phi PZ_i^0dx}{\int_{\Omega}e^Wdx}
-\dfrac{\int_{\Omega}e^W\phi dx\int_{\Omega}e^W PZ_i^0dx}{(\int_{\Omega}e^Wdx)^2} \right)=O(\lambda).
\end{equation*}
From (\ref{pz1}) and the above estimates, one has
\begin{equation}\label{limitofphi}
\lim_{\lambda\to 0}|\log \lambda|\int_{\tilde{\Omega}_i}e^w\tilde{\phi}_i dy=0.
\end{equation}

Next we multiply equation (\ref{linearproblem1}) by $Pw_i$ and integrate over $\Omega$,
\begin{equation*}
\begin{aligned}
&\int_{\Omega}\nabla \phi \nabla Pw_idx-\sum_j \int_{\Omega}e^{w_j}\phi Pw_idx-\rho^+\left(\dfrac{\int_{\Omega}e^W\phi Pw_idx}{\int_{\Omega}e^Wdx}
-\dfrac{\int_{\Omega}e^W\phi dx\int_{\Omega}e^WPw_idx}{(\int_{\Omega}e^Wdx)^2} \right)\\
&=\int_{\Omega}\nabla h\nabla Pw_idx.
\end{aligned}
\end{equation*}
Now we estimate the above equation term by term.
\begin{equation*}
\int_{\Omega}\nabla \phi \nabla Pw_i dx=\int_{\Omega}e^{w_i}\phi dx =\int_{\tilde{\Omega}_i}e^w\tilde{\phi}_idy=o(1)
\end{equation*}
by (\ref{limit1}) and the fact that
\begin{equation*}
\int_{\R^2}e^w\frac{1-|y|^2}{1+|y|^2}dy=0.
\end{equation*}
By the expansion of $Pw_i$,
\begin{equation*}
\begin{split}
&\sum_j\int_{\Omega}e^{w_j}\phi Pw_idx=\int_{\Omega}e^{w_i}\phi Pw_idx+\sum_{j\neq i}\int_{\Omega}e^{w_j}\phi Pw_idx\\
&=\int_{\tilde{\Omega}_i}e^w\tilde{\phi}_i\Big(-4\log \delta_i-2\log (1+|y|^2)+8\pi H(\xi_i,\xi_i)+O(\delta_i|y|+\delta_i^2)\Big)dy\\
&\quad +\sum_{j\neq i}\int_{\tilde{\Omega}_j}e^w \tilde{\phi}_j(8\pi G(\xi_i,\xi_j)+O(\delta_j|y|+\delta_j^2))dy\\
&=\gamma_i\int_{\R^2}e^w\frac{1-|y|^2}{1+|y|^2}[-2\log (1+|y|^2)]dy+o(1).
\end{split}
\end{equation*}
Moreover,
\begin{equation*}
\rho^+\left(\dfrac{\int_{\Omega}e^W\phi Pw_idx}{\int_{\Omega}e^Wdx}
-\dfrac{\int_{\Omega}e^W\phi dx\int_{\Omega}e^WPw_idx}
{(\int_{\Omega}e^Wdx)^2}\right)=o(1)
\end{equation*}
and
\begin{equation*}
\int_{\Omega}\nabla h\nabla Pw_idx=O(\|h\|_p\|Pw_i\|)=O(\log \lambda)^{\frac{1}{2}}\|h\|=o(1).
\end{equation*}
Combining all the above estimates, we have
\begin{equation*}
\gamma_i\int_{\R^2}e^w\frac{1-|y|^2}{1+|y|^2}[-2\log (1+|y|^2)]dy=0,
\end{equation*}
which implies that $\gamma_i=0$ since
\begin{equation*}
\int_{\R^2}e^w\frac{1-|y|^2}{1+|y|^2}[-2\log (1+|y|^2)]dy\neq 0.
\end{equation*}

\noindent{\bf Step 3.} Finally, we derive a contradiction.

\medskip

Multiply equation (\ref{linearproblem1} ) by $\phi$ and integrate:
\begin{equation*}
\begin{split}
\int_{\Omega}|\nabla \phi|^2dx-\sum_i\int_{\Omega}e^{w_i}\phi^2dx-\rho^+\left( \dfrac{\int_{\Omega}e^W\phi^2dx}{\int_{\Omega}e^Wdx}-\dfrac{(\int_{\Omega} e^W\phi dx)^2}{(\int_{\Omega}e^Wdx)^2}\right)=\int_{\Omega}\nabla h\nabla \phi dx.
\end{split}
\end{equation*}
From the estimates in step 1-2 and the assumptions on $\phi$ and $h$, it is not difficult to show that the left hand side of the above equation tends to $1$, while the right hand side has limit $0$. This is a contradiction which concludes the proof.
\end{proof}

Now we can derive a priori estimates for problem (\ref{linearproblem}).
\begin{proposition}\label{aprioribound}
Let $\mathcal{C}\subset \mathcal{F}_k\Omega$ be a compact set. Then, there exist $\lambda_0>0$ and $C>0$ such that for any $\lambda\in (0, \lambda_0)$, ${\bm\xi}\in \mathcal{C}$ and $h\in H_0^1(\Omega)$, if $(\phi, c_{ij})$ is a solution of \eqref{linearproblem}, we have
\begin{equation*}
\|\phi\|\leq C|\log \lambda|\|h\|.
\end{equation*}
\end{proposition}
\begin{proof}
By Lemma \ref{aprioriestimate} and (\ref{estimateofkernel}) , we know that
\begin{equation*}
\|\phi\|\leq C|\log \lambda|\left(\|h\|+\sum_{ij}|c_{ij}|\|PZ_i^j\|\right)\leq C|\log \lambda|\left(\|h\|+\sum_{ij}\frac{1}{\sqrt{\lambda}}|c_{ij}|\right).
\end{equation*}
In order to estimate $c_{ij}$,  multiply the equation (\ref{linearproblem}) by $PZ_i^j$ and integrating over $\Omega$,
\begin{equation*}
\begin{split}
&\int_{\Omega}\phi e^{w_i}(PZ_i^j-Z_i^j)dx+\sum_{\ell\neq i}\int_{\Omega}e^{w_\ell}\phi PZ_i^jdx
+O\left(\int_{\Omega}|\phi||PZ_i^j|dx+\int_{\Omega}|\phi|\int_{\Omega}|PZ_i^j|dx\right)\\
&=\int_{\Omega}\nabla h\nabla PZ_i^j+c_{ij}\int_{\Omega} e^{w_i}Z_i^jPZ_i^j dx+\sum_{k\neq i, \ell\neq j}o\left(\frac{|c_{k\ell}|}{\lambda}\right),
\end{split}
\end{equation*}
where in the last line we use (\ref{estimateofkernel1}). Since
\begin{equation*}
\begin{split}
\int_{\Omega}\phi e^{w_i}(PZ_i^j-Z_i^j)dx+\sum_{\ell\neq i}\int_{\Omega}e^{w_\ell}\phi PZ_i^jdx&=O(\|\phi\|(\|e^{w_i}\|_q+\|e^{w_\ell}PZ_i^j\|_q))=O\left(\lambda^{\frac{1-q}{q}}\|\phi\|\right),\\
O\left(\int_{\Omega}|\phi||PZ_i^j|+\int_{\Omega}|\phi|\int_{\Omega}|PZ_i^j|\right)&=O(\|PZ_i^j\|_2\|\phi\|)=O\left(|\log\lambda|^{\frac12}\|\phi\|\right),\\
\int_{\Omega}\nabla h\nabla PZ_i^j&=\|h\|\|PZ_i^j\|=O\left(\frac{1}{\sqrt{\lambda}}\|h\|\right),\\
\end{split}
\end{equation*}
we have
\begin{equation*}
|c_{ij}|+o\left(\sum_{k\neq j,\ell\neq i}|c_{k,\ell}|\right)=O\left(\lambda^{\frac{1}{q}}\|\phi\|
+\lambda|\log\lambda|^{\frac12}\|\phi\|+\lambda^{\frac{1}{2}}\|h\|\right).
\end{equation*}
Summing all $|c_{ij}|$ up and choosing suitable $q\in (1, 2)$, we can get that
\begin{equation*}
\|\phi\|\leq C|\log \lambda |\|h\|.
\end{equation*}
\end{proof}

From the above a priori estimate and the Fredholm alternative it is then standard to derive the following existence result, see for example Proposition 4.5 in \cite{dpr}.
\begin{proposition} \label{pro-ex}
Let $\mathcal{C}\subset \mathcal{F}_k\Omega$ be a compact set. Then, there exist $\lambda_0>0$ and $C>0$ such that for any $\lambda\in (0, \lambda_0)$, ${\bm\xi}\in \mathcal{C}$ and $h\in H^1_0(\Omega)$, there exists a unique solution $(\phi, c_{ij})$ of \eqref{linearproblem}, which satisfies
\begin{equation*}
\|\phi\|\leq C|\log \lambda|\|h\|.
\end{equation*}
\end{proposition}

\smallskip

\subsection{Nonlinear Problem}\label{sec3.4}

The aim of this subsection is to find $(\phi, \{c_{ij}\})$ such that $u=W+\phi$ solves
\begin{equation*}
\left\{\begin{array}{l}
\Delta u+\rho^+\dfrac{e^u}{\int_{\Omega}e^udx}-\lambda e^{-u}=\sum_{ij}c_{ij}e^{w_i}Z_i^j, \vspace{0.2cm}\\
\int_{\Omega}\nabla \phi \nabla PZ_i^jdx=0, \quad j=1,2,~i=1,\cdots,k.
\end{array}
\right.
\end{equation*}
The latter can be rewritten as
\begin{equation}\label{nonlinearproblem}
\begin{cases}
&\Delta \phi+\rho^+\Big(\dfrac{e^W\phi}{\int_\Omega e^Wdx}-\dfrac{e^W\int_\Omega e^W\phi dx}{(\int_\Omega e^Wdx)^2}  \Big)+\sum_{i=1}^k e^{w_i}\phi=\Big(R+\mathcal{S}(\phi)+\mathcal{N}(\phi)\Big)+\sum_{ij}c_{ij}e^{w_i}Z_i^j, \vspace{0.2cm}\\
&\int_{\Omega}\nabla \phi \nabla PZ_i^jdx=0, \quad j=1,2,~i=1,\cdots,k.
\end{cases}
\end{equation}
where $R$ is the error term defined in Subsection \ref{sec3.2} and
\begin{equation*}
\begin{split}
\mathcal{N}(\phi)&=\lambda \Big(f(W+\phi)-f(W)-f'(W)\phi \Big)+\rho^+\Big( g(W+\phi)-g(W)-g'(W)\phi\Big),\\
\mathcal{S}(\phi)&=\left(\sum_{i=1}^k e^{w_i}+\lambda f'(W)\right)\phi,
\end{split}
\end{equation*}
and
\begin{equation} \label{f,g}
f(W)=e^{-W}, \ g(W)=\frac{e^W}{\int_{\Omega}e^W dx}.
\end{equation}
Once the linear theory is carried out, the existence of a solution to the nonlinear problem \eqref{nonlinearproblem} follows a standard strategy. Observe that \eqref{nonlinearproblem} resembles the linear problem \eqref{linearproblem}. Therefore, the idea is to use the existence result of the linear problem, see Proposition \ref{pro-ex}, to construct a contraction map, knowing that the term $R+\mathcal{S}(\phi)+\mathcal{N}(\phi)$ is small. We omit here the details referring to Proposition 4.10 in \cite{dpr} for the full argument.
\begin{proposition}\label{existence}
Let $\mathcal{C}\subset \mathcal{F}_k\Omega$ be compact set. For any $\ve>0$ sufficiently small, there exist $\lambda_0>0$ and $C>0$ such that for any $\lambda\in (0, \lambda_0)$ and ${\bm\xi}\in \mathcal{C}$, there exists a unique $(\phi, \{c_{ij}\})$ satisfying \eqref{nonlinearproblem} and
\begin{equation*}
\|\phi\|\leq C\lambda^{\frac{1}{2}-\ve}, \  \|\partial_{\xi_i^j}\phi\|\leq C\lambda^{-\ve}, \ |c_{ij}|\leq C\lambda.
\end{equation*}
\end{proposition}

\subsection{The reduced problem}\label{sec3.5}
We introduce here the finite-dimensional reduction. In the previous subsection we have found a solution $u=W+\phi$ to the problem
\begin{equation*}
\begin{cases}
\Delta u+\rho^+\dfrac{e^u}{\int_{\Omega}e^udx}-\lambda e^{-u}=\sum_{ij}c_{ij}
e^{w_i}Z_i^j，\vspace{0.2cm}\\
\int_{\Omega}\nabla \phi \nabla PZ_i^jdx=0,~j=1,2,~i=1,\cdots,k.
\end{cases}
\end{equation*}
Consider now the associated energy functional:
\begin{equation}\label{energy}
J(u)=\frac{1}{2}\int_{\Omega}|\nabla u|^2dx-\rho^+\log \int_{\Omega}e^udx-\lambda \int_{\Omega}e^{-u}dx
\end{equation}
and let $\tilde{J}({\bm\xi})=J(W_{\bm\xi}+\phi_{\bm\xi})$.
\begin{lemma}\label{reduction}
Let ${\bm\xi}\in \mathcal{F}_k\Omega$ be a critical point of $\tilde{J}$, then for $\lambda$ small, $u=W_{{\bm\xi}}+\phi_{{\bm\xi}}$ is a solution of \eqref{mainproblem}.
\end{lemma}

\begin{proof}
If ${\bm\xi}$ is a critical point of $\tilde{J}(\bm\xi)$, then one has
\begin{equation*}
\langle J'(u), \partial_{\bm\xi}(W_{\bm\xi}+\phi_{\bm\xi})\rangle=0,
\end{equation*}
which is equivalent to
\begin{equation}\label{cij}
\langle \sum_{ij}c_{ij}e^{w_i}Z_i^j, \partial_{\xi_\ell^s}(W_{\bm\xi}+\phi_{\bm\xi})\rangle=0 \quad \mbox{ for }\ell=1,\cdots,k, \ s=1,2.
\end{equation}
Since
\begin{equation*}
\begin{split}
\int_{\Omega}e^{w_i}Z_i^j \partial_{\xi_\ell^s}\phi_{\bm\xi}dx&=\|e^{w_i}Z_i^j\|_q\|\partial_{\xi}\phi\|=O\left(\lambda^{\frac{2-3q}{2q}-\ve}\right)=o\left(\frac{1}{\lambda}\right),\\
\int_{\Omega}e^{w_i}Z_i^j \partial_{\xi_\ell^s}W_{\bm\xi}dx&=-\int_{\Omega}PZ_\ell^s e^{w_i}Z_i^jdx+O\left(\frac{1}{\sqrt{\lambda}}\right)=\frac{a}{\lambda}\delta_{i\ell}\delta_{js}+o\left(\frac{1}{\lambda}\right),
\end{split}
\end{equation*}
we conclude that
\begin{equation*}
c_{ij}+o(1)\sum_{\ell\neq i, s\neq j}c_{\ell s}=0,
\end{equation*}
which implies that all $c_{ij}$ are zero. So the corresponding $u$ is a solution of \eqref{mainproblem} as desired.
\end{proof}

Recall the definition of $\Lambda$ in (\ref{Lambda}). We next consider the expansion of the energy. 
\begin{proposition}\label{pro1}
It holds
\begin{equation*}
J(W)=\Lambda({\bm\xi})-8\pi k\log \lambda -(16\pi-24\pi \log 2)k+o(1),
\end{equation*}
$\mathcal{C}^1$ uniformly in ${\bm\xi}$ in compact sets of $\Omega$.
\end{proposition}
\begin{proof}
By the definition of $J(W)$ and $W$, one has
\begin{equation*}
\begin{split}
J(W)&=\frac{1}{2}\int_{\Omega}\left(|\nabla z|^2+\sum_{i=1}^k |\nabla Pw_i|^2-2\sum_{i=1}^k \nabla Pw_i \nabla z+2\sum_{i\neq j}\nabla Pw_i \nabla Pw_j\right)dx\\
&\quad-\rho^+\log\int_{\Omega}e^Wdx-\lambda \int_{\Omega}e^{-W}dx.
\end{split}
\end{equation*}
Using (\ref{eq2}),
\begin{equation*}
\begin{split}
\frac{1}{2}\int_{\Omega}|\nabla z|^2dx-\rho^+\log \int_{\Omega}e^Wdx
=\frac{1}{2}\int_{\Omega}|\nabla z|^2dx
-\rho^+\log \int_{\Omega} h(x,{\bm\xi})e^{z(x,{\bm\xi})}dx+O(\lambda).
\end{split}
\end{equation*}
While using (\ref{eq1}) and the estimate for $E_1$,
\begin{equation*}
\begin{split}
\lambda \int_{\Omega}e^{-W}dx
&=\sum_{i=1}^k\int_{\Omega}e^{w_i}dx+o(1)=8k\pi+o(1),\\
\int_{\Omega}\nabla Pw_i \nabla z dx&=\int_{\Omega} e^{w_i}z(x,{\bm\xi})dx
=\int_{\tilde{\Omega}_i}\frac{8}{(1+|y|^2)^2}z(\delta_i y+\xi_i, {\bm\xi})dy=8\pi z(\xi_i,{\bm\xi})+o(1),
\end{split}
\end{equation*}
where $\tilde{\Omega}_i=(\Omega-\xi_i)/\delta_i$. Moreover, using the expansion (\ref{bubble-projection})
\begin{equation*}
\begin{split}
\int_{\Omega}|\nabla Pw_i|^2dx&=\int_{\Omega}e^{w_i}Pw_idx\\
&=\int_{\Omega}e^{w_i}\Big(\log \frac{1}{(\delta_i^2+|x-\xi_i|^2)^2}+8\pi H(x,\xi_i)+O(\lambda) \Big)dx\\
&=64\pi^2 H(\xi_i,\xi_i)-2\int_{\tilde{\Omega}_i}e^{w_i}(\log \delta_i^2+\log (1+|y|^2))dy+o(1)\\
&=64\pi^2 H(\xi_i,\xi_i)-16\pi \log \delta_i^2-16\pi+o(1)\\
&=64\pi^2 H(\xi_i,\xi_i)-16\pi \log \frac{\lambda d_i({\bm\xi})}{8}-16\pi +o(1)\\
&=-64\pi^2 H(\xi_i,\xi_i)-128\pi^2\sum_{j\neq i}G(\xi_i,\xi_j)+16\pi z(\xi_i,{\bm\xi})-16\pi \log \lambda-16\pi+48\pi \log 2+o(1),
\end{split}
\end{equation*}
and for $i\neq j$,
\begin{equation*}
\begin{split}
\int_{\Omega}\nabla Pw_i \nabla Pw_jdx&=\int_{\Omega}e^{w_i}\Big(\log \frac{1}{(\delta_j^2+|x-\xi_j|^2)^2}+8\pi H(x,\xi_j)+O(\lambda) \Big)dx=64\pi^2 G(\xi_i,\xi_j)+o(1).
\end{split}
\end{equation*}
Combining all the above estimates, we have
\begin{equation*}
\begin{split}
J(W)&=\frac{1}{2}\int_{\Omega}|\nabla z|^2dx-\rho^+\log \int_{\Omega} h(x,{\bm\xi})e^{z(x,{\bm\xi})}dx-8\pi k\log \lambda\\
&\quad-32\pi^2\sum_{i=1}^k\left(H(\xi_i,\xi_i)+\sum_{j\neq i}G(\xi_i,\xi_j)\right)
-(16\pi-24\pi \log 2)k+o(1)\\
&=\Lambda({\bm\xi})-8\pi k\log \lambda -(16\pi-24\pi \log 2)k+o(1).
\end{split}
\end{equation*}
Next, we consider the derivative of $J(W)$.
\begin{equation*}
\begin{split}
\partial_{\xi_i^j} J(W)&=\int_{\Omega}\left(-\Delta W-\rho^+\dfrac{e^W}{\int_{\Omega} e^Wdx}+\lambda e^{-W}\right)\partial_{\xi_i^j}Wdx=-\int_{\Omega}(E_1(x)+E_2(x))\partial_{\xi_i^j}Wdx \\
&=4\int_{\Omega}E_1(z)Z_i^j dx+o(1)
=4\int_{\Omega}\left(\sum_\ell e^{w_\ell} -\lambda e^{-W}\right)Z_i^jdx,
\end{split}
\end{equation*}
where $E_1, E_2$ were introduced in Lemma \ref{estimateoferror} and where we used
\begin{equation*}
\partial_{\xi_i^j}W=-4PZ_i^j+O(1).
\end{equation*}
Using the definition of $w_i$ and $Z_i^j$, for $\ell\neq i$
\begin{equation*}
\begin{split}
\int_{\Omega} e^{w_\ell}Z_i^jdx&=\int_{\Omega}\frac{8\delta_\ell^2}{(\delta_\ell^2+|x-\xi_\ell|^2)^2}\frac{x_j-\xi_i^j}{\delta_i^2+|x-\xi_i|^2}dx=8\pi \frac{\xi_\ell^j-\xi_i^j}{|\xi_\ell-\xi_i|^2}+o(1).
\end{split}
\end{equation*}
Moreover, taking $\eta>0$ such that $|\xi_i-\xi_j|\geq2\eta$ and $d(\xi_i,\partial\Omega)\geq2\eta$, we have
\begin{equation*}
\begin{split}
\int_{B(\xi_\ell, \eta)}\lambda e^{-W}Z_i^jdx &=\lambda\int_{B(\xi_\ell,\eta)}\exp\Big[8\pi\sum_i H(x,\xi_i)-z(x,{\bm\xi})+O(\lambda)\Big]\frac{x_j-\xi_i^j}{\delta_i^2+|x-\xi_i|^2}
\prod_{i=1}^k\frac{1}{(\delta_i^2+|x-\xi_i|^2)^2}dx\\
&=\frac{\lambda}{\delta_\ell^2}\int_{\tilde{\Omega}_\ell}\exp\Big[8\pi H(\xi_\ell,\xi_\ell)+8\pi\sum_{j\neq \ell}G(\xi_\ell,\xi_j)-z(\xi_\ell, {\bm\xi})\Big]\frac{1}{(1+|y|^2)^2}\frac{\xi_\ell^j-\xi_i^j}{|\xi_\ell-\xi_i|^2}dx+o(1)\\
&=8\pi \frac{\xi_\ell^j-\xi_i^j}{|\xi_\ell-\xi_i|^2}+o(1).
\end{split}
\end{equation*}
Let $$\gamma(x, {\bm\xi})=8\pi H(x, \xi_i)+8\pi \sum_{j\neq i}G(x, \xi_j)-z(x,{\bm\xi}).$$
Then, 
\begin{equation*}
\begin{split}
\int_{B(\xi_i, \eta)}\lambda e^{-W}Z_i^jdx
&=\lambda\int_{B(\xi_i,\eta)}\exp\Big[8\pi\sum_i H(x,\xi_i)-z(x,{\bm\xi})+O(\lambda)\Big]\frac{x_j-\xi_i^j}{\delta_i^2+|x-\xi_i|^2}
\prod_{i=1}^k\frac{1}{(\delta_i^2+|x-\xi_i|^2)^2}dx\\
&=\frac{\lambda}{\delta_i^3}\int_{\tilde{\Omega}_i}\frac{1}{(1+|y|^2)^2}\frac{y_j}{1+|y|^2}\exp\Big[8\pi H(\xi_i+\delta_i y, \xi_i)+8\pi \sum_{j\neq i}G(\xi_i+\delta_i y,\xi_j)\\
&\quad-z(\xi_i+\delta_i y, {\bm\xi})\Big]dy+o(1)\\
&=\frac{8}{\delta_i}\int_{B(0, \frac{\eta}{\delta_i})}\frac{y_j}{(1+|y|^2)^3}exp[\gamma(\xi_i+\delta_i y, {\bm\xi})-\gamma(\xi_i,{\bm\xi})]dy+o(1)\\
&=\frac{8}{\delta_i}\int_{\R^2}\frac{y_j}{(1+|y|^2)^3}\frac{\partial \gamma}{\partial x}(\xi_i, {\bm\xi})\cdot \delta_i y \,dy +o(1)\\
&=2\pi \frac{\partial \gamma}{\partial x}(\xi_i, {\bm\xi})+o(1).
\end{split}
\end{equation*}
Finally,
\begin{equation*}
\begin{split}
\int_{\Omega\setminus\bigcup_i B(\xi_i, \eta)}\lambda e^{-W}Z_i^jdx\leq C\lambda\int_{\Omega\setminus\bigcup_i B(\xi_i, \eta)}e^{\sum_{\ell}Pw_\ell}|Z_i^j|dx
\leq C\lambda =o(1).
\end{split}
\end{equation*}
Combining the above estimates, we have
\begin{equation*}
\begin{split}
\partial_{\xi_i^j}J(W)=-8\pi\frac{\partial \gamma}{\partial x}(\xi_i,{\bm\xi})+o(1)=\partial_{\xi_i^j}\Lambda({\bm\xi})+o(1),
\end{split}
\end{equation*}
as desired, where we used \eqref{der}.
\end{proof}

Finally, we have the following expansion of the reduced energy.
\begin{proposition}\label{pro2}
It holds
\begin{equation*}
\tilde{J}({\bm\xi}):=J(W_{\bm\xi}+\phi_{\bm\xi})=J(W_{\bm\xi})+o(1),
\end{equation*}
$\mathcal{C}^1$ uniformly in ${\bm\xi}$ in compact sets of $\mathcal{F}_k\Omega$.
\end{proposition}

\begin{proof}
To simplify the notation, we shall drop the sub-index ${\bm\xi}$ in the proof. It is not difficult to show that
\begin{equation*}
\begin{split}
J(W+\phi)-J(W)&=\frac{1}{2}\int_{\Omega}|\nabla \phi|^2dx+\int_{\Omega}\nabla W\nabla \phi dx+\lambda \int_{\Omega}e^{-W}(1-e^{-\phi})dx\\
&\quad+\rho^+\Big(\log \int_{\Omega}e^Wdx-\log \int_{\Omega}e^{W+\phi}\Big)dx\\
&=-\int_{\Omega}\Delta z(x,{\bm\xi})\phi dx-\rho^+\int_{\Omega} \dfrac{h(x,{\bm\xi})e^{z(x,{\bm\xi})}\phi}{\int_{\Omega}h(x, {\bm\xi})e^{z(x,{\bm\xi})}dx}dx+\int_{\Omega}\sum_ie^{w_i}\phi dx-\lambda\int_{\Omega}e^{-W}\phi dx\\
&\quad+\rho^+\Big( \log \int_{\Omega}e^Wdx-\log \int_{\Omega}e^{W+\phi}dx
+ \int_{\Omega}\dfrac{h(x,{\bm\xi})e^{z(x,{\bm\xi})}\phi}{\int_{\Omega}h(x, {\bm\xi})e^{z(x,{\bm\xi})}dx}dx\Big)\\
&\quad+\lambda \int_{\Omega}e^{-W}(1-e^{-\phi}+\phi)dx+\|\phi\|^2=o(1).
\end{split}
\end{equation*}
Next we consider the derivatives.
\begin{equation*}
\begin{split}
\partial_{\xi_i^j}[J(W+\phi)-J(W)]
&=-\int_{\Omega}\left(\Delta(W+\phi)+\rho^+\dfrac{e^{W+\phi}}
{\int_{\Omega}e^{W+\phi}dx}
-\lambda e^{-(W+\phi)}\right)\partial_{\xi_i^j}\phi dx\\
&\quad-\int_{\Omega}\Big[ \Delta \phi +\rho^+\left(\frac{e^{W+\phi}}{\int_{\Omega} e^{W+\phi}dx}
-\frac{e^W}{\int_{\Omega}e^Wdx}\right)-\lambda (e^{-(W+\phi)}-e^{-W}) \Big]\partial_{\xi_i^j}Wdx\\
&=~\sum_{i,j}\int_{\Omega}c_{ij}e^{w_i}Z_i^j \partial_{\xi_i^j}\phi dx
-\int_{\Omega}\Delta\phi \partial_{\xi_i^j}W dx
-\int_{\Omega}\lambda e^{-W}\phi \partial_{\xi_i^j}W dx\\
&\quad+\int_{\Omega}\lambda(e^{-(W+\phi)}-e^{-W}+e^{-W}\phi)\partial_{\xi_i^j}W dx\\
&\quad+\rho^+\int_{\Omega}\left(\frac{e^{W+\phi}}{\int e^{W+\phi}}-\frac{e^W}{\int e^W}\right)\partial_{\xi_i^j}W dx.
\end{split}
\end{equation*}
Using the estimate for $c_{ij}$ in Proposition \ref{existence}, we have
\begin{equation*}
\begin{split}
\sum_{i,j}\int_{\Omega}c_{ij}e^{w_i}Z_i^j \partial_{\xi_i^j}\phi dx&=\sum_{i,j} |c_{ij}|\|\partial_{\xi_i^j}\phi\| \cdot \|e^{w_i}Z_i^j\|_q=O(\lambda^{\frac{2-3q}{2q}+1-\ve})=o(1),
\end{split}
\end{equation*}
provided $q$ is sufficiently close to $1$. Recalling the definitions of $f,g$ in \eqref{f,g} we exploit now the estimates in \cite[Lemma 4.7]{dpr}. For some $\theta\in(0,1)$ and $p$ sufficiently close to $1$ we have
\begin{equation*}
\begin{split}
\int_{\Omega}\lambda(e^{-(W+\phi)}-e^{-W}+e^{-W}\phi)\partial_{\xi_i^j}Wdx&=\int_{\Omega}\lambda f''(W+\theta \phi)\phi^2 \partial_{\xi_i^j}Wdx=\|\lambda f''(W+\theta \phi)\phi^2\|_p\|\partial_{\xi_i^j}W\|_q \\ &=O(\lambda^{\frac{1-pq}{pq}-\frac{1}{2}+1-2\ve})=o(1).
\end{split}
\end{equation*}
Moreover, for some $\tilde\theta\in(0,1)$ and suitable $p,q$
\begin{equation*}
\begin{split}
\rho^+\int_{\Omega}\left(\frac{e^{W+\phi}}{\int e^{W+\phi}}-\frac{e^W}{\int e^W}\right)\partial_{\xi_i^j}W\,dx
&=\rho^+\int_{\Omega}g'(W+\tilde\theta \phi)\phi \partial_{\xi_i^j}W\,dx
=\|g'(W+\tilde\theta \phi)\phi \|_p \|\partial_{\xi_i^j}W\|_q \\ &=O(\lambda^{\frac{1}{2}-\ve})=o(1).
\end{split}
\end{equation*}
Recall that
\begin{equation*}
\lambda e^{-W}=\sum_{i=1}^k e^{w_i}+O(\lambda)\quad  \mbox{ and }\quad \partial_{\xi_i^j}W=-4PZ_i^j+O(1),
\end{equation*}
for ${\bm\xi}$ in compact sets of $\mathcal{F}_k\Omega$. Then
\begin{equation*}\begin{split}
\lambda\int_{\Omega}e^{-W}\phi \partial_{\xi_i^j}Wdx&=-4\sum_{\ell=1}^k \int_{\Omega}e^{w_\ell}\phi PZ_{i}^j dx+o(1)\\
&=-4\int_{\Omega}e^{w_i}Z_i^j\phi dx-4\sum_{\ell \neq i}\int_{\Omega}e^{w_\ell}\phi Z_i^jdx+o(1)\\
&=-4\int_{\Omega}\nabla \phi \nabla PZ_i^j dx+o(1)=o(1)
\end{split}
\end{equation*}
by the orthogonality condition satisfied by $\phi$. Moreover, again by the orthogonality condition we have
\begin{equation*}
\begin{split}
\int_{\Omega}\Delta \phi \partial_{\xi_i^j}Wdx&=-\int_{\Omega}\nabla \phi\nabla \partial_{\xi_i^j}Wdx
=-4\int_{\Omega}\nabla \phi (\nabla PZ_i^j+O(1))dx \\ &=O(1)\int_{\Omega}|\nabla \phi |dx=o(1).
\end{split}
\end{equation*}
Combining the above estimates, we have
\begin{equation*}
\partial_{\xi_i^j}\tilde{J}({\bm\xi})=\partial_{\xi_i^j}J(W)+o(1),
\end{equation*}
as desired.
\end{proof}

\medskip
\noindent{\bf Proof of Theorem \ref{thm}. } Let $\mathcal{K}\subset \mathcal{F}_k\Omega$ be a $C^1$-stable set of critical points of $\Lambda$. Then, by Propositions \ref{pro1}-\ref{pro2}, for $\lambda>0$ small, there exists ${\bm\xi}_\lambda$ critical point of $\tilde{J}$ and $d({\bm\xi}_\lambda, \mathcal{K})\to 0$ as $\lambda\to 0$. By Lemma \ref{reduction}, $u_\lambda=W+\phi$ is a solution of (\ref{mainproblem}). It follows that $u_\lambda$ solves the original problem (\ref{sinh-gordon}) with $\rho^+_\lambda=\rho^+$ and
\begin{equation*}
\rho_\lambda^-=\lambda\int_{\Omega}e^{-u}dx=\lambda \int_{\Omega}e^{-W}dx+o(1)=8k\pi+o(1).
\end{equation*}
\qed

\

\section{Asymmetric blow up}\label{sec4}
\subsection{Approximate solutions}\label{sec4.1}
In this section we will derive the proof of Theorem \ref{thm-1}. To this end we will always assume that $\Omega$ is $l-$symmetric for $l\geq 2$ even according to \eqref{symmetry1}. Therefore, we will consider symmetric functions such that
\begin{equation}
\label{symmetry1}
u(x)=u (\mathcal{R}_l\cdot x),
\end{equation}
see \eqref{symmetry1}, and define
\begin{equation*}
\mathcal{H}_l:=\left\{u\in H_0^1(\Omega), \ u \mbox{ satisfies }(\ref{symmetry1})\right\}.
\end{equation*}

Consider problem (\ref{mainproblem}) and let $k\geq 2$ be an odd integer. In order to construct blow up solutions with local masses $(4\pi k(k-1), 4\pi k(k+1))$, we need to consider the following singular Liouville equation. Let $\alpha\geq 2$. It is known that
\begin{equation*}
w_\delta^\alpha(x)=\log \frac{2\alpha^2\delta^\alpha}{(\delta^\alpha+|x|^\alpha)^2}, \quad \delta>0,
\end{equation*}
solves the problem
\begin{equation*}
\Delta w+|x|^{\alpha-2}e^w=0~\mathrm{in}~\R^2, \quad \int_{\R^2}|x|^{\alpha-2}e^wdx<\infty.
\end{equation*}
Similarly to the previous section, let $Pu$ be the projection of the function $u$ into $H_0^1(\Omega)$. We look here for a sign changing solution of the form
\begin{equation*}
u=W+\phi(x), \quad W(x)=z(x)+\sum_{i=1}^k(-1)^i Pw_i(x),
\end{equation*}
where $\phi$ is a small error term, $z(x)$ is a solution of (\ref{equationofz-1}) and $Pw_i=Pw_{\delta_i}^{\alpha_i}$ with
\begin{equation}\label{deltai}
\alpha_i=4i-2, \quad \delta_i=d_i\lambda^{\frac{k-i+1}{4i-2}}, \, d_i>0, \quad i=1,\cdots,k.
\end{equation}
The latter parameters are chosen such that the interaction of different bubbles is small. More precisely, the following functions will play an important role in the interaction estimate:
\begin{equation*}
\begin{split}
\Theta_i(y)=~&Pw_i(\delta_i y)-w_i(\delta_i y)-(\alpha_i-2)\log|\delta_iy|+\sum_{j\neq i}(-1)^{j-i}Pw_j-z(\delta_i y)+\log \lambda, \quad \mbox{i odd},
\end{split}
\end{equation*}
\begin{equation*}
\begin{split}
T_i(y)=~&Pw_i(\delta_i y)-w_i(\delta_i y)-(\alpha_i-2)\log|\delta_iy|+\sum_{j\neq i}(-1)^{j-i}Pw_j+z(\delta_i y)-\log Q, \quad \mbox{i even},
\end{split}
\end{equation*}
where
\begin{equation}\label{qi}
Q=\rho_0^{-1}\int_{\Omega}e^{z-8k\pi G(x,0)}dx.
\end{equation}
As we will see in the sequel, in order to make these two functions small, we will need to choose $\delta_i$ and $\alpha_i$ such that
\begin{equation}\label{parameter1}
(\alpha_i-2)+\sum_{j<i}(-1)^{j-i}2\alpha_j=0, \quad i=1,\cdots,k,
\end{equation}
and
\begin{equation}\label{parameter2}
\begin{split}
&-\alpha_i\log\delta_i-\log(2\alpha_i^2)
-2\sum_{j>i}(-1)^{j-i}\alpha_j\log\delta_j-z(0)+\sum_{j=1}^k(-1)^{j-i}h_j(0)+\log\lambda=0, \quad \mbox{i odd},
\end{split}
\end{equation}
\begin{equation}\label{parameter3}
\begin{split}
&-\alpha_i\log\delta_i-\log(2\alpha_i^2)
-2\sum_{j>i}(-1)^{j-i}\alpha_j\log\delta_j+z(0)+\sum_{j=1}^k(-1)^{j-i}h_j(0)-\log Q=0, \quad \mbox{i even},
\end{split}
\end{equation}
$h_i(x)=4\pi\alpha_iH(x,0).$
From (\ref{parameter1}) we deduce that $\alpha_1=2$ and $\alpha_i=\alpha_{i-1}+4$ for $i\geq 2$ which implies the choice of $\alpha_i$ in (\ref{deltai}). On the other hand, from (\ref{parameter2}) and (\ref{parameter3}) one easily deduces that
\begin{equation*}
\delta_k^{\alpha_k}=\lambda e^{\sum_j(-1)^{j-k}h_j(0)-z(0)-\log(2\alpha_k^2)}
=\lambda e^{8k\pi H(0,0)-z(0)-\log(2\alpha_k^2)}.
\end{equation*}
Moreover,
\begin{equation*}
\delta_{i-1}^{\alpha_{i-1}}=\dfrac{\delta_i^{\alpha_i}}{4\alpha_i^2\alpha_{i-1}^2Q}\lambda.
\end{equation*}
From the above identities, one can get that
\begin{equation*}
\delta_i=d_i\lambda^{\frac{k-i+1}{4i-2}},
\end{equation*}
for some $d_i>0$, which implies (\ref{deltai}).

\medskip

We estimate now $\Theta_i$ and $T_i$. First, using the maximum principle it is not difficult to see that 
\begin{equation*}
\begin{split}
Pw_i(x)&=w_i(x)-\log(2\alpha_i^2\delta_i^{\alpha_i})+h_i(x)+O(\delta_i^{\alpha_i})\\
&=-2\log(\delta_i^{\alpha_i}+|x|^{\alpha_i})+h_i(x)+O(\delta_i^{\alpha_i})
\end{split}
\end{equation*}
and for $i,j=1,\cdots,k$,
\begin{equation}\label{expansion-proj}
Pw_i(\delta_j y)=
\begin{cases}
-2\alpha_i\log(\delta_j|y|)+h_i(0)
+O\Big(\frac{1}{|y|^{\alpha_i}}\Big(\frac{\delta_i}{\delta_j}\Big)^{\alpha_i}\Big)\\
+O(\delta_j|y|)+O(\delta_i^{\alpha_i})\quad \mbox{ if }~i<j\\
\\
-2\alpha_i\log\delta_i-2\log(1+|y|^{\alpha_i})+h_i(0)\\
+O(\delta_i(y))+O(\delta_i^{\alpha_i})\quad \mbox{ if }~i=j\\
\\
-2\alpha_i\log\delta_i+h_i(0)
+O\Big(|y|^{\alpha_i}\Big(\frac{\delta_j}{\delta_i}\Big)^{\alpha_i}\Big)\\
+O(\delta_j|y|)+O(\delta_i^{\alpha_i})
\quad \mbox{ if }~i>j.
\end{cases}
\end{equation}
where $h_i(x)=4\pi \alpha_iH(x,0)$.
\begin{remark}\label{rem1}
From the above expansion, one can get that for $|x|\geq \delta_0$ for $\delta_0>0$ small, the following expansion holds:
\begin{equation*}
\begin{aligned}
\sum_{i=1}^k(-1)^iPw_i(x)=4\pi\sum_i(-1)^i\alpha_iH(x,0)-2\sum_i(-1)^i\alpha_i\log|x|+O(\delta_k^{\alpha_k})
\end{aligned}
\end{equation*}
From the definition of $\alpha_i$ we have $\sum_{i=1}^k(-1)^i\alpha_i=(-1)^k2k$ and hence, for $k$ odd it holds
\begin{equation*}
\sum_{i=1}^k(-1)^i Pw_i(x)=-8kG(x,0)+O(\lambda).
\end{equation*}
\end{remark}

We next introduce the following shrinking annulus 
\begin{equation} \label{annulus}
A_j=\left\{x\in \Omega, \sqrt{\delta_{j-1}\delta_j}\leq |x|\leq \sqrt{\delta_j\delta_{j+1}}\right\}, \quad j=1,\cdots,k,
\end{equation}
where $\delta_0:=0$ and $\delta_{k+1}:=+\infty$.
\begin{lemma}\label{choiceofparameter}
For any $y\in \frac{A_i}{\delta_i}$, the following estimates hold:
\begin{align}
&\Theta_i(y)=O(\delta_i|y|+\lambda ), \quad \mbox{i odd}, \label{estimate1}\\ 
&T_i(y)=O(\delta_i|y|+\lambda ), \quad \mbox{i even}. \label{estimate2}
\end{align}
In particular,
\begin{equation}\label{estimate3}
\sup_{y\in\frac{A_i}{\delta_i}}|\Theta_i(y)|+\sup_{y\in\frac{A_i}{\delta_i}}|T_i(y)|=O(1).
\end{equation}
\end{lemma}
\begin{proof}
Consider $y\in \frac{A_i}{\delta_i}$. From (\ref{expansion-proj}), and using (\ref{parameter1}) and (\ref{parameter2}), for $i$ odd,
\begin{equation*}
\begin{split}
\Theta_i(y)=~&-\alpha_i\log\delta_i-\log(2\alpha_i^2)+h_i(0)-(\alpha_i-2)\log|\delta_iy|+O(\delta_i|y|+\delta_i^{\alpha_i})\\
&+\sum_{j<i}(-1)^{i-j}\Big[-2\alpha_j\log(\delta_i|y|)+h_j(0)+O\Big(\frac{1}{|y|^{\alpha_j}}\Big(\frac{\delta_j}{\delta_i}
\Big)^{\alpha_j}\Big)+O(\delta_i|y|+\delta_j^{\alpha_j})\Big]\\
&+\sum_{j>i}(-1)^{j-i}\Big[-2\alpha_j\log\delta_j+h_j(0)+O\Big(|y|^{\alpha_j}\Big(\frac{\delta_i}{\delta_j}
\Big)^{\alpha_j}\Big)+O(\delta_i|y|+\delta_j^{\alpha_j})\Big]\\
&-z(0)+\log\lambda+O(\delta_i|y|)\\
=~&\Big[\sum_{j=1}^k(-1)^{j-i}h_j(0)-\alpha_i\log\delta_i-\log(2\alpha_i^2)-2\sum_{j>i}(-1)^{j-i}\alpha_j\log\delta_j-z(0)+\log\lambda\Big]\\
&\ \ \ (=0 \mbox{ because \ of }(\ref{parameter2}))\\
&-\log|\delta_i|y||\Big[(\alpha_i-2)+\sum_{j<i}(-1)^{i-j}2\alpha_j\Big]\\
&\ \ \ (=0 \mbox{ because \ of }(\ref{parameter1}))\\
&+O(\delta_i|y|)+\sum_j\delta_j^{\alpha_j}+\sum_{j>i}O\Big(|y|^{\alpha_j}\Big(\frac{\delta_i}{\delta_j}
\Big)^{\alpha_j}\Big)+\sum_{j<i}O\Big(\frac{1}{|y|^{\alpha_j}}\Big(\frac{\delta_j}{\delta_i}
\Big)^{\alpha_j}\Big)\\
=~&O(\delta_i|y|)+\sum_j\delta_j^{\alpha_j}+\sum_{j>i}O\Big(|y|^{\alpha_j}\Big(\frac{\delta_i}{\delta_j}
\Big)^{\alpha_j}\Big)+\sum_{j<i}O\Big(\frac{1}{|y|^{\alpha_j}}\Big(\frac{\delta_j}{\delta_i}
\Big)^{\alpha_j}\Big)\\
=~&O(\delta_i|y|+\lambda).
\end{split}
\end{equation*}
Similarly, for $i$ even,
\begin{equation*}
\begin{split}
T_i(y)=~&\Big[\sum_{j=1}^k(-1)^{j-i}h_j(0)-\alpha_i\log\delta_i-\log(2\alpha_i^2)-2\sum_{j>i}(-1)^{j-i}\alpha_j\log\delta_j+z(0)-\log Q\Big]\\
&\ \ \ (=0 \mbox{ because \ of }(\ref{parameter3}))\\
&-\log|\delta_i|y||\Big[(\alpha_i-2)+\sum_{j<i}(-1)^{i-j}2\alpha_j\Big]\\
&\ \ \ (=0 \mbox{ because \ of }(\ref{parameter1}))\\
&+O(\delta_i|y|)+\sum_j\delta_j^{\alpha_j}+\sum_{j>i}O\Big(|y|^{\alpha_j}\Big(\frac{\delta_i}{\delta_j}
\Big)^{\alpha_j}\Big)+\sum_{j<i}O\Big(\frac{1}{|y|^{\alpha_j}}\Big(\frac{\delta_j}{\delta_i}
\Big)^{\alpha_j}\Big)\\
=~&O(\delta_i|y|)+\sum_j\delta_j^{\alpha_j}+\sum_{j>i}O\Big(|y|^{\alpha_j}\Big(\frac{\delta_i}{\delta_j}
\Big)^{\alpha_j}\Big)+\sum_{j<i}O\Big(\frac{1}{|y|^{\alpha_j}}\Big(\frac{\delta_j}{\delta_i}
\Big)^{\alpha_j}\Big)\\
=~&O(\delta_i|y|+\lambda).
\end{split}
\end{equation*}
Finally, (\ref{estimate3}) follows from the above two estimates since $\delta_i|y|=O(1)$ when $y\in\frac{A_i}{\delta_i}$.
\end{proof}

Finally, we will need the following non-degeneracy result for entire singular Liouville equations which was derived in \cite[Theorem 6.1]{grossi-pistoia} for $l=2$ and which can be extended to any $l\geq2$ even.
\begin{proposition}\label{nondegeneracyofmeanfield}
Assume $\phi : \R^2\to \R$ satisfying (\ref{symmetry1}) is a solutions of
\begin{equation*}
\Delta \phi+2\alpha^2\frac{|y|^{\alpha-2}}{(1+|y|^\alpha)^2}\phi=0 \quad\mbox{ in }\quad\R^2, ~\quad \int_{\R^2}|\nabla \phi|^2dy<\infty,
\end{equation*}
with $\alpha\geq 2$ and $\frac{\alpha}{2}$ odd. Then,
\begin{equation*}
\phi(y)=\gamma\frac{1-|y|^\alpha}{1+|y|^\alpha}, \quad \mbox{for some $\gamma\in \R$}.
\end{equation*}
\end{proposition}

\smallskip

\subsection{Estimate of the error term}\label{sec4.2}
In this subsection we estimate the error of the approximate solution. To this end, set
\begin{equation*}
\begin{aligned}
E_1=~&\rho^+\dfrac{e^W}{\int_{\Omega}e^W dx}-\sum_{i\ even}|x|^{\alpha_i-2}e^{w_i}-\rho_0\dfrac{e^{z-8k\pi G(x,0)}}{\int_{\Omega }e^{z-8k\pi G(x,0)} dx},\\ 
E_2=~&\lambda e^{-W}-\sum_{i \ odd}|x|^{\alpha_i-2}e^{w_i}.
\end{aligned}
\end{equation*}
\begin{lemma}\label{errorestimate}
For any $q\geq 1$ sufficiently close to $1$, the following holds:
\begin{equation*}
\|E_1\|_{q}=O\Big(\lambda^{\frac{2-q}{2q(2k-1)}}\Big), \quad \|E_2\|_q=O\Big(\lambda^{\frac{2-q}{2q(2k-1)}}\Big).
\end{equation*}
\end{lemma}
\begin{proof}
First we consider $E_2$. Recall the definition of the annulus $A_i$ in \eqref{annulus}.
\begin{equation*}
\begin{split}
\int_{\Omega}E_2^qdx&=\sum_{i=1}^k \int_{A_i}E_2^qdx
=\sum_{i \ odd}\int_{A_i}E_2^qdx+\sum_{i \ even}\int_{A_i}E_2^qdx=I_1+I_2.
\end{split}
\end{equation*}
One has
\begin{equation*}
\begin{split}
I_1&=\sum_{i \ odd}\int_{A_i}E_2^qdx=\sum_{i \ odd}\int_{A_i}|\lambda e^{\sum_{l\ odd}Pw_l-\sum_{l \ even}Pw_l-z}-\sum_{j \ odd}|x|^{\alpha_j-2}e^{w_j}|^qdx\\
&\leq C\sum_{i \ odd}\int_{A_i}||x|^{\alpha_i-2}e^{w_i}-\lambda e^{\sum_{l \ odd}Pw_l-\sum_{l \ even}Pw_l-z}|^qdx+C\sum_{i, j\ odd,\ i\neq j}\int_{A_i}||x|^{\alpha_j-2}e^{w_j}|^qdx\\
&=I_{11}+I_{12}.
\end{split}
\end{equation*}

Let us estimate $I_{11}$. For fixed $i$ odd,
\begin{equation*}
\begin{split}
&\int_{A_i}||x|^{\alpha_i-2}e^{w_i}-\lambda e^{\sum_{l \ odd}Pw_l-\sum_{l \ even}Pw_l-z}|^qdx\\
&=\int_{A_i}|x|^{q(\alpha_i-2)}e^{qw_i}|1-e^{Pw_i-w_i-(\alpha_i-2)\log |x|+\sum_{j\neq i \ odd}Pw_j-\sum_{l \ even}Pw_l-z+\log\lambda}|^qdx\\
&=C\delta_i^{2-2q}\int_{\frac{A_i}{\delta_i}}\frac{|y|^{q(\alpha_i-2)
}}{(1+|y|^{\alpha_i})^{2q}}|1-e^{\Theta_i(y)}|^qdy=C\delta_i^{2-2q}\int_{\frac{A_i}{\delta_i}}\frac{|y|^{q(\alpha_i-2)
}}{(1+|y|^{\alpha_i})^{2q}}|\Theta_i(y)|^qdy \\
&~\quad(\mbox{ using }(\ref{estimate1}))\\
&=O\Big(\delta_i^{2-2q}\int_{\frac{A_i}{\delta_i}}\frac{|y|^{q(\alpha_i-2)
}}{(1+|y|^{\alpha_i})^{2q}}|\delta_i |y|+\lambda|^qdy \Big)=O(\delta_i^{2-2q}\lambda^q+\delta_i^{2-q})=O(\delta_1^{2-2q}\lambda^q+\delta_k^{2-q})\\
&=O(\lambda^{q+k(1-q)}+\lambda^{\frac{2-q}{2(2k-1)}})
=O(\lambda^{\frac{2-q}{2(2k-1)}}),
\end{split}
\end{equation*}
provided that $q$ is close to $1$. Therefore, we get $I_{11}=O(\lambda^{\frac{2-q}{2(2k-1)}}).$ 

For $I_{12}$, fix $j\neq i$ odd,
\begin{equation}\label{I12}
\begin{split}
&\int_{A_i}||x|^{\alpha_j-2}e^{w_j}|^qdx=C\int_{A_i}\left(\frac{|x|^{\alpha_j-2}\delta_j^{\alpha_j}}{(\delta_j^{\alpha_j}
+|x|^{\alpha_j})^2}\right)^qdx=C\delta_j^{2-2q}\int_{\frac{\sqrt{\delta_{i-1}\delta_i}}{\delta_j}\leq |y|\leq \frac{\sqrt{\delta_i\delta_{i+1}}}{\delta_j}}\frac{|y|^{q(\alpha_j-2)
}}{(1+|y|^{\alpha_j})^{2q}}dy \\
&=\begin{cases}
O\left(\delta_j^{2-2q}\Big(\frac{\sqrt{\delta_i\delta_{i+1}}}{\delta_j}\Big)
^{(\alpha_j-2)q+2}\right)\quad &\mbox{ for }~j>i\\
\\
O\left(\delta_j^{2-2q}\Big(\frac{\sqrt{\delta_i\delta_{i-1}}}{\delta_j}\Big)
^{-(\alpha_j+2)q+2}\right)\quad &\mbox{ for }~j<i
\end{cases}\\
&=\begin{cases}
O\Big(\delta_3^{2-2q}(\frac{\delta_{k-1}}{\delta_k})^{(\alpha_k-2)q+2}\Big)
=O\left(\lambda^{\frac{(k-2)(1-q)}{5}
+\frac{(2k+1)(2(k-1)q+1)}{4(k-1)^2-1}}\right)\\
\\
O\Big(\delta_1^{2-2q}(\frac{\delta_{k-2}}{\delta_{k-1}})^{q(2+\alpha_{k-2})-2}\Big)
=O\left(\lambda^{k(1-q)+\frac{(2k+1)(2(k-2)q-1)}{4(k-2)^2-1}}\right)
\end{cases}\\
&=O\Big(\lambda^{\frac{2-q}{2(2k-1)}}\Big).
\end{split}
\end{equation}
provided that $q$ is close to $1$. Therefore, $\|I_1\|_q=O\Big(\lambda^{\frac{2-q}{2q(2k-1)}}\Big)$.

\smallskip

Next, let us estimate $I_2$. For $l$ even fixed,
\begin{equation*}
\begin{split}
&\int_{A_l}E_2^qdx\leq C\int_{A_l}|\lambda e^{-W}|^qdx+C\sum_{i \ odd}\int_{A_l}||x|^{\alpha_i-2}e^{w_i}|^qdx=I_{21}+I_{22}.
\end{split}
\end{equation*}
We have,
\begin{equation}\label{I21}
\begin{split}
I_{21}&=C\int_{A_l}|\lambda e^{-Pw_l-\sum_{j\neq l \ even}Pw_j-z+\sum_{i \ odd}Pw_i}|^qdx=C\lambda^q \delta_l^2\int_{\frac{A_l}{\delta_l}}|e^{-w_l(\delta_l y)-(\alpha_l-2)\log|\delta_l y|-T_l(y)-\log Q}|^qdy\\
&~\quad(\mbox{ using }(\ref{estimate2}))\\
&=O\Big(\delta_l^{2+2q}\lambda^q
\int_{{\sqrt{\frac{\delta_{l-1}}{\delta_l}}\leq |y|\leq \sqrt{\frac{\delta_{l+1}}{\delta_l}}}
}
\frac{(1+|y|^{\alpha_l})^{2q}}{|y|^{(\alpha_l-2)q}}(1+\delta_l|y|+\lambda)^qdy \Big)\\
&=O\Big(\delta_l^{2+2q}\lambda^q \Big[\Big(\frac{\delta_{l+1}}{\delta_l}\Big)^{\frac{(\alpha_l+2)q}{2}+1}
+\Big(\frac{\delta_{l}}{\delta_{l-1}}\Big)^{\frac{(\alpha_l-2)q}{2}-1}
\Big]\Big)=O\Big(\delta_{2}^{2+2q}\lambda^q\Big[\Big(\frac{\delta_{3}}{\delta_2}\Big)^{\frac{(\alpha_2+2)q}{2}+1}
+\Big(\frac{\delta_{2}}{\delta_{1}}\Big)^{\frac{(\alpha_2-2)q}{2}-1}
\Big] \Big)\\
&=O\Big(\lambda^{q+\frac{(k-1)(1+q)}{3}-\frac{(2k+1)(2q-1)}{6}}\Big)=O\Big(\lambda^{\frac{2-q}{2(2k-1)}}\Big),
\end{split}
\end{equation}
if $q$ is close to $1$. Moreover, similarly to the estimate of $I_{12}$, one can also get that $I_{22}=O\Big(\lambda^{\frac{2-q}{2(2k-1)}}\Big).$

Combining all the above estimates, one has
\begin{equation}
\int_{\Omega}E_2^qdx=O(\lambda^{\frac{2-q}{2(2k-1)}}).
\end{equation}

Next we consider $E_1$. First we need to estimate $\int_{\Omega}e^{W}dx$. For $i$ even fixed,
\begin{equation*}
\begin{split}
\int_{A_i}e^Wdx&=\int_{A_i}e^{Pw_i-w_i+z+\sum_{j\neq i}(-1)^{j-i}Pw_j-(\alpha_i-2)\log|x|}|x|^{\alpha_i-2}e^{w_i}dx=\int_{\frac{A_i}{\delta_i}}e^{T_i(y)+\log Q}|\delta_i y|^{\alpha_i-2}e^{w_i(\delta_i y)}\delta_i^2dy\\
&=\int_{\frac{A_i}{\delta_i}}e^{\log Q+O(\delta_i |y|+\lambda)}|\delta_i y|^{\alpha_i-2}e^{w_i(\delta_i y)}\delta_i^2dy=4\pi\alpha_iQ+O(\lambda^{\frac{1}{2(2k-1)}}),
\end{split}
\end{equation*}
where we have used Lemma \ref{choiceofparameter} for the estimate of $T_i(y)$ and the fact that
\begin{equation*}
\int_{\R^2}\frac{2\alpha_i^2|y|^{\alpha_i-2}}{(1+|y|^{\alpha_i})^2}
dy=4\pi\alpha_i.
\end{equation*}
For $i<k$ odd and fixed, reasoning as in (\ref{I21}) with $q=1$, one has 
\begin{equation*}
\begin{split}
\int_{A_i}e^Wdx=\int_{A_i}e^{-Pw_i-\sum_{j\neq i}(-1)^{j-i}Pw_j+z}dx
=O(\lambda^{\frac{2k-5}{6}}).
\end{split}
\end{equation*}
Finally for $i=k$ which is odd, using Remark \ref{rem1},
\begin{equation*}
\begin{split}
\int_{A_k}e^Wdx&=\int_{A_k}e^z e^{-Pw_k-\sum_{j\neq k}(-1)^{j-k}Pw_j}dx=\int_{|x|>\sqrt{\delta_{k-1}\delta_k}}e^{z-8k\pi G(x,0)}dx
+O(\delta_k^{\alpha_k})+O(\lambda^{\frac{1}{2(2k-1)}})\\
&=\int_{\Omega}e^{z-8k\pi G(x,0)}dx+O(\lambda^{\frac{1}{2(2k-1)}}).
\end{split}
\end{equation*}
In conclusion, one has
\begin{equation}\label{integralofew}
\begin{split}
\int_{\Omega}e^Wdx&=\int_{\Omega}e^{z-8k\pi G(x,0)}dx+\sum_{i \ even}4\pi\alpha_iQ+O(\lambda^{\frac{1}{2(2k-1)}})=\frac{\rho^+}{\rho_0}\int_{\Omega}e^{z-8k\pi G(x,0)}dx+O(\lambda^{\frac{1}{2(2k-1)}}),
\end{split}
\end{equation}
where we used the definition of $Q$ in (\ref{qi}) and the fact that
\begin{equation*}
\sum_{i \ even}4\pi\alpha_iQ=\frac{\rho^+-\rho_0}{\rho_0}\int_{\Omega}e^{z-8k\pi G(x,0)}dx,
\end{equation*}
since $\sum_{i\ even}4\pi\alpha_i=4\pi k(k-1)=\rho^+-\rho_0$.
\medskip

With the estimate for $\int_{\Omega}e^Wdx$ in hand, we now consider $E_1$.
\begin{equation*}
\begin{split}
\int_{\Omega}E_1^qdx&=\sum_{i\ even}\int_{A_i}E_1^qdx+\sum_{l\ odd}\int_{A_l}E_1^qdx=J_1+J_2.
\end{split}
\end{equation*}
First for $i$ even fixed,
\begin{equation*}
\begin{split}
\int_{A_i}E_1^qdx&=\int_{A_i}|\rho^+\frac{e^W}{\int_{\Omega}e^Wex}-\rho_0\frac{e^{z-8k\pi G(x,0)}}{\int_{\Omega}e^{z-8k\pi G(x,0)}dx}-
|x|^{\alpha_i-2}e^{w_i}-\sum_{j\neq i \ even}|x|^{\alpha_j-2}e^{w_j}|^qdx\\
&\leq C\int_{A_i}|\rho^+\frac{e^W}{\int_{\Omega}e^Wex}-|x|^{\alpha_i-2}e^{w_i}|^qdx+C\int_{A_i}|\rho_0\frac{e^{z-8k\pi G(x,0)}}{\int_{\Omega}e^{z-8k\pi G(x,0)}dx}|^qdx
+C\sum_{j\neq i \ even}\int_{A_i}||x|^{\alpha_j-2}e^{w_j}|^qdx\\
&=C\int_{A_i}|\rho^+\frac{e^W}{\int_{\Omega}e^Wdx}-|x|^{\alpha_i-2}e^{w_i}|^qdx
+O(\lambda^{\frac{2-q}{2(2k-1)}})+O(\delta_{i+1}^{4kq+2})\\
&=C\delta_i^{2-2q}\int_{\frac{A_i}{\delta_i}}\frac{|y|^{(\alpha_i-2)q}}{(1+|y|^{\alpha_i})^{2q}}\left|1-e^{Pw_i(\delta_i y)-w_i(\delta_i y)-(\alpha_i-2)\log|\delta_i y|+\sum_{j\neq i}(-1)^{j-i}Pw_j+z+\log \frac{\rho^+}{\int_{\Omega}e^Wdx}}\right|^qdx\\
&~\quad(\mbox{ by }(\ref{estimate2}))\\
&=C\delta_i^{2-2q}\int_{\frac{A_i}{\delta_i}}\frac{|y|^{(\alpha_i-2)q}}{(1+|y|^{\alpha_i})^{2q}}
\left|1-e^{T_i(y)+\log Q +\log\frac{\rho_0}{\int_{\Omega}e^{z-8k\pi G(x,0)}dx}+O(\lambda^{\frac{1}{2(2k-1)}})}\right|^qdx\\
&=C\delta_i^{2-2q}\int_{\frac{A_i}{\delta_i}}\frac{|y|^{(\alpha_i-2)q}}{(1+|y|^{\alpha_i})^{2q}}
|\delta_i|y|+O(\lambda^{\frac{1}{2(2k-1)}})|^qdy=O(\lambda^{\frac{2-q}{2(2k-1)}}).
\end{split}
\end{equation*}
So we have
\begin{equation}
J_1=O\Big(\lambda^{\frac{2-q}{2(2k-1)}}\Big).
\end{equation}

Next, consider $J_2$. For $l<k$ odd and fixed, similarly to the estimates in (\ref{I21}), (\ref{I12}) and using (\ref{integralofew})
\begin{equation*}
\begin{split}
\int_{A_l}|E_1|^qdx&=O(1)\Big(\int_{A_l}|e^{-Pw_l-\sum_{j\neq l}(-1)^{j-l}Pw_j+z}|^qdx+\int_{A_l}\left|\frac{e^{z-8k\pi G(x,0)}}{\int_{\Omega}e^{z-8k\pi G(x,0)}dx}\right|^qdx
\\
&\quad+\sum_{j\ even}\int_{A_l}||x|^{\alpha_j-2}e^{w_j}|^qdx\Big)
=O(\lambda^{\frac{2-q}{2(2k-1)}}).
\end{split}
\end{equation*}
Finally, we consider the case $l=k$ which is odd: using (\ref{integralofew}) and (\ref{I12})
\begin{equation*}\begin{split}
\int_{A_k}E_1^qdx&\leq C\int_{A_k}\left|\rho^+\dfrac{e^{z+\sum_{i}(-1)^iPw_i}}{\int_{\Omega} e^Wdx}-\rho_0\dfrac{e^{z-8k\pi G(x,0)}}{\int_{\Omega}e^{z-8k\pi G(x,0)}dx}\right|^qdx+C\sum_{i\ even}\int_{A_k}|x|^{(\alpha_i-2)q}e^{qw_i}dx\\
&=C\int_{A_k}\left|\rho_0\frac{e^{z+\sum_{i}(-1)^iPw_i}}{\int_{\Omega}e^{z-8k\pi G(x,0)}dx}-\rho_0\frac{e^{z-8k\pi G(x,0)}}{\int_{\Omega}e^{z-8k\pi G(x,0)}dx}\right|^qdx
+O(\lambda^{\frac{2-q}{2(2k-1)}})\\
&=O(\delta_k^{2q})+O(\lambda^{\frac{2-q}{2(2k-1)}})
=O(\lambda^{\frac{2-q}{2(2k-1)}}).
\end{split}\end{equation*}

In conclusion, one has
\begin{equation*}
\|E_1\|_q=O\Big(\lambda^{\frac{2-q}{2q(2k-1)}}\Big).
\end{equation*}
\end{proof}

\subsection{The linear theory}\label{sec4.3}

In this subsection, we consider the linear problem: given $h\in \mathcal{H}_l$, we look for $\phi\in\mathcal{H}_l$ such that
\begin{equation}\label{linearproblem-1}
\Delta \phi+\rho^+\left(\dfrac{e^W\phi}{\int_\Omega e^Wdx}-\dfrac{e^W\int_\Omega e^W\phi dx}{(\int_\Omega e^Wdx)^2}\right)+\lambda e^{-W}\phi=\Delta h\quad \mbox{in}~\Omega.
\end{equation}
First we have the following apriori estimate:
\begin{lemma}\label{aprioriestimate-1}
There exist $\lambda_0>0$ and $C>0$ such that for any $\lambda\in (0,\lambda_0)$, $h\in\mathcal{H}_l$ and $\phi\in \mathcal{H}_l$ solution of \eqref{linearproblem-1} we have
\begin{equation*}
\|\phi\|\leq C|\log \lambda| \|h\|.
\end{equation*}
\end{lemma}

\smallskip

We start by listing some straightforward integrals which will be useful in the proof of Lemma \ref{aprioriestimate-1}.
\begin{lemma}
The following hold:
\begin{equation}\label{integral1}
\int_{\R^2}\frac{|y|^{\alpha_i-2}}{(1+|y|^{\alpha_i})^2}\frac{1-|y|^{\alpha_i}}{1+|y|^{\alpha_i}}dy=0,
\end{equation}
\begin{equation}\label{integral2}
\int_{\R^2}2\alpha_i^2\frac{|y|^{\alpha_i-2}}{(1+|y|^{\alpha_i})^2}\frac{1-|y|^{\alpha_i}}{1+|y|^{\alpha_i}}\log (1+|y|^{\alpha_i})^2dy=-4\pi \alpha_i,
\end{equation}
\begin{equation}\label{integral3}
\int_{\R^2}2\alpha_i^2\frac{|y|^{\alpha_i-2}}{(1+|y|^{\alpha_i})^2}\frac{1-|y|^{\alpha_i}}{1+|y|^{\alpha_i}}\log |y|dy=-4\pi.
\end{equation}
\end{lemma}

\medskip
\noindent{\bf Proof of Lemma \ref{aprioriestimate-1}.}
We prove it by contradiction. Assume there exist $\lambda_n\to 0$, $h_n\in \mathcal{H}_l$ and $\phi_n \in \mathcal{H}_l$ which solves (\ref{linearproblem-1}) such that
\begin{equation*}
\|\phi_n\|=1, \quad \ |\log \lambda_n|\|h_n\|\to 0 \quad \mbox{ as }n\to \infty.
\end{equation*}
In the following, we omit the index $n$ for simplicity.
For $i=1, \cdots, k$, define $\tilde{\phi}_i(y)$ as
\begin{equation*}
\begin{split}
\tilde{\phi}_i(y)=\begin{cases}
\phi_i(\delta_i y), \quad &y\in \tilde{\Omega}_i=\frac{\Omega}{\delta_i},\\
0,  &y\in \R^2\setminus \tilde{\Omega}_i.
\end{cases}
\end{split}
\end{equation*}

\noindent{\bf Step 1. } We claim that
\begin{equation}\label{limit2-1}
\phi\to 0~\mbox{ weakly \ in \ $H_0^1(\Omega)$\ and \ strongly \ in \ $L^q(\Omega)$ \ for $q\geq 2$}.
\end{equation}
and
$$\tilde{\phi}_i \mbox{ is \ bounded\ in }H_{\alpha_i}(\R^2)$$

\medskip

Letting $\psi\in C_0^\infty(\Omega\setminus \{0\})$ and multiplying equation (\ref{linearproblem-1}) by $\psi$ and integrating, one has
\begin{equation}\label{boundnessofphi}
\begin{split}
-\int_{\Omega}\nabla \psi \nabla \phi dx+\int_{\Omega}\lambda e^{-W}\phi\psi dx+\rho^+\left(\dfrac{\int_{\Omega}e^W\phi\psi dx}{\int_{\Omega}e^Wdx}-\dfrac{\int_{\Omega}e^W\phi dx\int_{\Omega} e^W\psi dx}{(\int_{\Omega}e^Wdx)^2} \right)=\int_{\Omega }\Delta h \psi dx.
\end{split}
\end{equation}
By the assumption on $\phi$, using the fact that in compact sets of $\Omega\setminus\{0\}$,
$$e^W=e^{z(x)-8k\pi G(x,0)}+O(\lambda)\quad \mbox{and} \quad
\lambda e^{-W}=O(\lambda),$$
one has
\begin{equation*}
\phi \to \phi^* \mbox{ weakly \ in \ $H_0^1(\Omega)$ \ and \ strongly \ in \ $L^q(\Omega)$ \ for \ $q\geq 2$}
\end{equation*}
where
\begin{equation*}
\begin{split}
-\int_{\Omega}\nabla \phi^* \nabla \psi dx+\rho^+\dfrac{\int_{\Omega} e^{z-8k\pi G(x,0)}\phi^*\psi dx}{\int_{\Omega}e^{z-8k\pi G(x,0)}}-\rho^+\dfrac{\int_{\Omega}e^{z-8k\pi G(x,0)}\psi dx\int_{\Omega}e^{z-8k\pi G(x,0)}\phi^*dx}{(\int_{\Omega}e^{z-8k\pi G(x,0)}dx)^2}=0.
\end{split}
\end{equation*}
So $\|\phi^*\|_{H_0^1(\Omega)}\leq 1$ and it solves
\begin{equation*}
\Delta \phi^*+\rho^+\left(\dfrac{e^{z-8k\pi G(x,0)}\phi^*}{\int_{\Omega}e^{z-8k\pi G(x,0)}dx}-\dfrac{e^{z-8k\pi G(x,0)}\int_{\Omega}e^{z-8k\pi G(x,0)}\phi^*dx}{(\int_{\Omega}e^{z-8k\pi G(x,0)}dx)^2}\right)=0.
\end{equation*}
By the non-degeneracy of $z(x)$ we get $\phi^*=0$.  Thus (\ref{limit2-1}) is proved.

\medskip

Now we prove that $\tilde{\phi}_i$ is bounded in $H_{\alpha_i}(\R^2)$. First it is easy to check that
\begin{equation}\label{hbound}
\int_{\R^2}|\nabla \tilde{\phi}_i|^2dy=\int_{\Omega}|\nabla \phi_i|^2dx\leq 1 \quad \mbox{ for }i=1,\cdots,k.
\end{equation}
We multiply (\ref{linearproblem-1}) again by $\phi$ and integrate,
\begin{equation}\label{eq3}
\int_{\Omega}|\nabla \phi|^2dx-\int_{\Omega}\lambda e^{-W}\phi^2dx-\rho^+\Big(\frac{\int_{\Omega}e^W\phi^2dx}{\int_{\Omega}e^Wdx}
-\frac{(\int_{\Omega}e^W\phi dx)^2}{(\int_{\Omega}e^Wdx)^2} \Big)=\int_{\Omega}\nabla h\nabla \phi dx.
\end{equation}
From the above equation, one can get that,
\begin{equation*}
\begin{split}
\int_{\Omega}\lambda e^{-W} \phi_i^2dx&\leq \int_{\Omega}|\nabla \phi|^2dx-\rho^+\Big(\frac{\int_{\Omega}e^W\phi^2dx}{\int_{\Omega}e^Wdx}
-\frac{(\int_{\Omega}e^W\phi dx)^2}{(\int_{\Omega}e^Wdx)^2} \Big)-\int_{\Omega}\nabla h\nabla \phi dx\\
&\leq 1+o(1)+\|h\|=O(1)
\end{split}
\end{equation*}
where we used (\ref{limit2-1}). Let $i$ be odd. Lemma \ref{errorestimate} gives 
\begin{equation*}
\int_{\Omega}|x|^{\alpha_i-2}e^{w_i}\phi^2dx\leq C,
\end{equation*}
or equivalently
\begin{equation*}
\int_{\R^2}\frac{|y|^{\alpha_i-2}}{(1+|y|^{\alpha_i})^2}\tilde{\phi}_i^2dy\leq C.
\end{equation*}
Combined with (\ref{hbound}), we deduce that $\tilde{\phi}_i$ is bounded in $H_{\alpha_i}(\R^2)$ when $i$ is odd.

We consider now the case for $i$ even. From (\ref{integralofew}), $e^W=e^{z-8k\pi G(x,0)}+O(\lambda)$ uniformly on compact sets of $\Omega\setminus\{0\}$ and recalling (\ref{limit2-1}), we get that
\begin{equation}\label{boundness1}
\int_{\Omega}e^W\phi dx=O(1).
\end{equation}
Moreover, by (\ref{eq3}) one can get that
\begin{equation}\label{boundness2}
\rho^+\left(\dfrac{\int_{\Omega}e^W\phi^2dx}{\int_{\Omega}e^Wdx}
-\dfrac{(\int_{\Omega}e^W\phi dx)^2}{(\int_{\Omega}e^Wdx)^2} \right)=O(1).
\end{equation}
Combining (\ref{boundness1}) and (\ref{boundness2}), we have
\begin{equation}\label{eq8}
\int_{\Omega}e^W\phi^2dx=O(1).
\end{equation}
By Lemma \ref{errorestimate}, (\ref{limit2-1}) and (\ref{eq8}), $\int_{\Omega}|x|^{\alpha_i-2}e^{w_i}\phi^2dx=O(1)$ for $i$ even, which implies that
\begin{equation*}
\int_{\R^2}\frac{|y|^{\alpha_i-2}}{(1+|y|^{\alpha_i})^2}\tilde{\phi}_i^2dy=O(1).
\end{equation*}
So we get that also for $i$ even, $\tilde{\phi}_i$ is bounded in $H_{\alpha_i}(\R^2)$.

\medskip

\noindent {\bf Step 2.}
 We claim that
  \begin{equation}\label{limit1-1}
\tilde{\phi}_i(y)\to \gamma_i\frac{1-|y|^2}{1+|y|^2} \mbox{ weakly \ in\  $H_{\alpha_i}(\R^2)$ \ and \ strongly \ in \ $L_{\alpha_i}(\R^2)$}, \gamma_i\in\mathbb{R}.
\end{equation}

\medskip

From Step 1, we know that $\tilde{\phi}_i\to \tilde{\phi}_i^*$ weakly  in  $H_{\alpha_i}(\R^2)$  and strongly  in  $L_{\alpha_i}(\R^2)$. Consider $\tilde{\psi}\in C_0^\infty(\R^2\setminus\{0\})$ and let $\mathcal{K}$ be its support. For $n$ large, one has
\begin{equation*}
\mathcal{K}\subset\frac{A_i}{\delta_i}=\left\{y\in \tilde{\Omega}_i, \sqrt{\frac{\delta_{i-1}}{\delta_i}}\leq |y|\leq \sqrt{\frac{\delta_{i+1}}{\delta_i}}\right\}.
\end{equation*}
Define $\psi_i=\tilde{\psi}(\frac{x}{\delta_i})$. Multiplying (\ref{linearproblem-1}) by $\psi_i$ and integrating over $\Omega$,
\begin{equation}
\label{psi_i-1}
\begin{split}
\int_{\Omega}\nabla \phi\nabla \psi_idx
-\rho^+\left( \dfrac{\int_{\Omega}e^W\phi\psi_idx}{\int_{\Omega}e^Wdx}
-\dfrac{\int_{\Omega}e^W\phi dx\int_{\Omega}e^W\psi_idx}{(\int_{\Omega}e^Wdx)^2} \right)
-\int_{\Omega}\lambda e^{-W}\phi \psi_idx
=\int_{\Omega}\nabla h\nabla \psi_idx.
\end{split}
\end{equation}
Consider first $i$ even. According to Lemma \ref{errorestimate}, one has
\begin{equation*}
\begin{split}
\rho^+\dfrac{\int_{\Omega}e^W\phi dx}{\int_{\Omega}e^Wdx}&=\sum_{j\ even}\int_{\Omega} |x|^{\alpha_j-2}e^{w_j}\phi dx+\rho_0\dfrac{\int_{\Omega}e^{z-8k\pi G(x,0)}\phi dx}{\int_{\Omega}e^{z-8k\pi G(x,0)}dx}+o(1)\\
&=\sum_{j\ even}\int_{\R^2}\frac{2\alpha_j^2|y|^{\alpha_j-2}\tilde{\phi}_j}{(1+|y|^{\alpha_i})^2}dy+\rho_0\dfrac{\int_{\Omega}e^{z-8k\pi G(x,0)}\phi dx}{\int_{\Omega}e^{z-8k\pi G(x,0)}dx}+o(1)\\
&=\sum_{j\ even}\int_{\R^2}\frac{2\alpha_j^2|y|^{\alpha_j-2}\tilde{\phi}_j^*}{(1+|y|^{\alpha_i})^2}dy+o(1)
\end{split}
\end{equation*}
where in the last line we used (\ref{limit2-1}). Similarly, one has
\begin{equation*}
\begin{split}
\rho^+\dfrac{\int_{\Omega}e^W\phi\psi_idx}{\int_{\Omega}e^Wdx}&=\int_{\R^2} \dfrac{2\alpha_i^2|y|^{\alpha_i-2}}{(1+|y|^{\alpha_i})^2}\tilde{\psi}\tilde{\phi}_i^*dy+o(1),\\
&\\
\rho^+\dfrac{\int_{\Omega}e^W\psi_idx}{\int_{\Omega}e^Wdx}&=\int_{\R^2} \frac{2\alpha_i^2|y|^{\alpha_i-2}}{(1+|y|^{\alpha_i})^2}\tilde{\psi}dy+o(1),\\
&\\
\lambda \int_{\Omega}e^{-W}\phi\psi_idx&=\sum_{j\ odd}\int_{\Omega} |x|^{\alpha_j-2}e^{w_j}\phi\psi_idx+o(1)=o(1).
\end{split}
\end{equation*}
Thus, $\tilde{\phi}_i^*$ satisfies
\begin{equation*}
\begin{aligned}
\int_{\R^2}\nabla \tilde{\phi}_i^*\nabla \tilde{\psi}_idy-\int_{\R^2}\dfrac{2\alpha_i^2|y|^{\alpha_i-2}}{(1+|y|^{\alpha_i})^2}\tilde{\phi}_i^*\tilde{\psi}dy=-\dfrac{1}{\rho^+}\Big(\int_{\R^2}\dfrac{2\alpha_i^2|y|^{\alpha_i-2}}{(1+|y|^{\alpha_i})^2}\tilde{\psi}dy \Big)\Big(\int_{\R^2} \dfrac{2\alpha_i^2|y|^{\alpha_i-2}}{(1+|y|^{\alpha_i})^2}\tilde{\phi}_i^*dy\Big).
\end{aligned}
\end{equation*}
From this we deduce that the function
\begin{equation*}
\tilde{\phi}_i^*-\frac{1}{\rho^+}\int_{\R^2}\frac{2\alpha_i^2|y|^{\alpha_i-2}}{(1+|y|^{\alpha_i})^2}\tilde{\phi}_i^*dy\in H_{\alpha_i}(\R^2)
\end{equation*}
is a solution of
\begin{equation}
\label{4.regular1}
\Delta \phi+\frac{2\alpha_i^2|y|^{\alpha_i-2}}{(1+|y|^{\alpha_i})^2}\phi=0\quad \mbox{ in }\quad \R^2\setminus\{0\}.
\end{equation}
Since $\int |\nabla \tilde{\phi}_i^*|^2dy\leq 1$, $\tilde{\phi}_i^*$ is a solution in the whole space $\R^2$. By Proposition \ref{nondegeneracyofmeanfield}, we get that $\tilde{\phi}_i^*-\frac{1}{\rho^+}\int_{\R^2}\frac{2\alpha_i^2|y|^{\alpha_i-2}}{(1+|y|^{\alpha_i})^2}\tilde{\phi}_i^*dy
=\gamma_i \frac{1-|y|^{\alpha_i}}{1+|y|^{\alpha_i}}$ for some $\gamma_i$. By(\ref{integral1}) one has
\begin{equation*}
\int_{\R^2}\dfrac{2\alpha_i^2|y|^{\alpha_i-2}}{(1+|y|^{\alpha_i})^2}\tilde{\phi}_i^*dy
=\dfrac{1}{\rho^+}\int_{\R^2}\frac{2\alpha_i^2|y|^{\alpha_i-2}}{(1+|y|^{\alpha_i})^2}dy\int_{\R^2}\dfrac{2\alpha_i^2|y|^{\alpha_i-2}}{(1+|y|^{\alpha_i})^2}\tilde{\phi}_i^*dy
\end{equation*}
which implies that
\begin{equation*}
\left(\frac{4\pi \alpha_i}{\rho^+}-1\right)
\int_{\R^2}\dfrac{2\alpha_i^2|y|^{\alpha_i-2}}
{(1+|y|^{\alpha_i})^2}\tilde{\phi}_i^*dy=0.
\end{equation*}
Since $\rho^+\neq 4\pi \alpha_i$ we deduce that
\begin{equation*}
\tilde{\phi}_i^*=\gamma_i\frac{1-|y|^{\alpha_i}}{1+|y|^{\alpha_i}}.
\end{equation*}
Hence, (\ref{limit1-1}) is proved for $i$ even.

We next turn to $i$ odd. In this case, we consider (\ref{psi_i-1}) with $i$ odd and estimate each term separately,
\begin{equation*}
\begin{split}
\int_{\Omega}e^W\psi_idx=o(1), \quad \int_{\Omega}e^W\phi\psi_idx=o(1),
\end{split}
\end{equation*}
and
\begin{equation*}
\begin{split}
\lambda\int_{\Omega} e^{-W}\phi\psi_idx=\int_{\Omega}|x|^{\alpha_i-2}e^{w_i}\phi\psi_idx+o(1)=\int_{\R^2}\frac{2\alpha_i^2|y|^{\alpha_i-2
}}{(1+|y|^{\alpha_i})^2}\tilde{\phi}_i\tilde{\psi}dy
+o(1).
\end{split}
\end{equation*}
Hence, $\tilde{\phi}_i^*$ satisfies
\begin{equation*}
\int_{\R^2}\nabla \tilde{\phi}_i^*\nabla \tilde{\psi}dy-\int_{\R^2}\frac{2\alpha_i^2|y|^{\alpha_i-2}}{(1+|y|^{\alpha_i})^2}\tilde{\phi}^*_i\tilde{\psi}dy
=0,
\end{equation*}
namely $\tilde{\phi}_i^*$ is a solution of
\begin{equation*}
\Delta\phi+\dfrac{2\alpha_i^2|y|^{\alpha_i-2}}{(1+|y|^{\alpha_i})^2}\phi=0\quad \mbox{ in }\quad \R^2\setminus\{0\},
\end{equation*}
and again we conclude by using Proposition \ref{nondegeneracyofmeanfield}.

\medskip

\noindent{\bf Step 3.} In this step, we will prove some estimates on the speed of convergence.
We set
\begin{equation}\label{sigma}
\sigma_i(\lambda):=|\log\lambda|\int_{\R^2}2\alpha_i^2
\frac{|y|^{\alpha_i-2}}{(1+|y|^{\alpha_i})^2}\tilde{\phi}_idy.
\end{equation}
We will show that
\begin{equation*}
\begin{cases}
\sigma_i(\lambda)=o(1)&\mbox{ for \ i\ odd}\\
\\
\sigma_i(\lambda)-\dfrac{4\pi\alpha_i}{\rho^+}\Big(\sum_{j\ even}\sigma_j(\lambda)+|\log\lambda|\rho_0\dfrac{\int_{\Omega}e^{z-8k\pi G(x,0)}\phi dx}{\int_{\Omega}e^{z-8k\pi G(x,0)}dx}\Big)=o(1)\quad &\mbox{ for \ i \ even}.
\end{cases}
\end{equation*}
Set $Z_i^0=\frac{\delta_i^{\alpha_i}-|x|^{\alpha_i}}{\delta_i^{\alpha_i}+|x|^{\alpha_i}}$, we know that $Z_i^0$ is a solution of
\begin{equation*}
\Delta Z+|x|^{\alpha_i-2}e^{w_i}Z=0\quad\mbox{ in }\quad \R^2.
\end{equation*}
Let $PZ_i^0$ be its the projection onto $H_0^1(\Omega)$, that is
\begin{equation*}
\Delta PZ_i^0+|x|^{\alpha_i-2}e^{w_i}Z_i^0=0 \ \mbox{ in } \ \Omega, \quad PZ_i^0=0 \ \mbox{ on } \ \partial\Omega .
\end{equation*}
By maximum principle one can show
\begin{equation}\label{projection}
PZ_i^0=Z_i+1+O(\delta_i^{\alpha_i})=\frac{2\delta_i^{\alpha_i}}{\delta_i^{\alpha_i}+|x|^{\alpha_i}}
+O(\delta_i^{\alpha_i}),
\end{equation}
which implies
\begin{equation}\label{projection1}
PZ_i^0(\delta_jy)=
\begin{cases}
O\left(\frac{1}{|y|^{\alpha_i}}(\frac{\delta_i}{\delta_j})^{\alpha_i}\right) +O(\delta_i^{\alpha_i})&\mbox{ for } i<j,\\
\\
\frac{2}{1+|y|^{\alpha_i}}+O(\delta_i^{\alpha_i}),\quad &\mbox{ for } i=j,\\
\\
2+O\left(|y|^{\alpha_i}(\frac{\delta_j}{\delta_i})^{\alpha_i}\right) +O(\delta_i^{\alpha_i})&\mbox{ for } i>j,
\end{cases}
\end{equation}
and
\begin{equation}\label{projection2}
\|PZ_i^0\|_q^q=O(\delta_i^2), \ q>1.
\end{equation}
First we consider $i$ even. Multiply (\ref{linearproblem-1}) by $PZ_i^0$ and integrate over $\Omega$,
\begin{equation}
\label{eq4}
\begin{split}
\int_{\Omega}\nabla \phi\nabla PZ_i^0dx
-\rho^+\left(\dfrac{\int_{\Omega}e^W\phi PZ_i^0dx}{\int_{\Omega}e^Wdx}
-\dfrac{\int_{\Omega}e^W\phi dx\int_{\Omega} e^W PZ_i^0dx}
{(\int_{\Omega}e^Wdx)^2} \right)
-\int_{\Omega}\lambda e^{-W}\phi PZ_i^0dx=-\int_{\Omega}\nabla h\nabla PZ_i^0dx.
\end{split}
\end{equation}
For the first term,
\begin{equation}\label{term1}
\begin{split}
\int_{\Omega}\nabla \phi \nabla PZ_i^0dx=-\int_{\Omega}\phi \Delta PZ_i^0dx
=\int_{\Omega}|x|^{\alpha_i-2}e^{w_i}\phi Z_i^0 dx.
\end{split}
\end{equation}
By Lemma \ref{errorestimate}, (\ref{limit2-1}), (\ref{projection}) and (\ref{projection2}),
\begin{equation}\label{term2}
\begin{split}
&\dfrac{\int_{\Omega}e^W\phi PZ_i^0dx}{\int_{\Omega}e^Wdx}=\sum_{j\ even}\int_{\Omega}|x|^{\alpha_j-2}e^{w_j}\phi PZ_i^0dx
+\rho_0\dfrac{\int_{\Omega}e^{z-8k\pi G(x,0)}PZ_i^0\phi dx}{\int _{\Omega}e^{z-8k\pi G(x,0)}dx}+o\Bigr(\frac{1}{|\log \lambda|}\Bigr)\\
&=\int_{\Omega}|x|^{\alpha_i-2}e^{w_i}\phi dx+\int_{\Omega}|x|^{\alpha_i-2}e^{w_i}\phi Z_i^0 dx+\sum_{j\neq i \ even}\int_{\Omega}|x|^{\alpha_j-2}e^{w_j}\phi PZ_i^0 dx +o\Bigr(\frac{1}{|\log \lambda|}\Bigr)
\end{split}
\end{equation}

For $j\neq i$,
\begin{equation}\label{eq5}
\begin{split}
&\int_{\Omega}|x|^{\alpha_j-2}e^{w_j}\phi PZ_i^0dx=
\int_{\tilde{\Omega}_j}
\frac{2\alpha_j^2|y|^{\alpha_j-2}}{(1+|y|^{\alpha_j})^2}\tilde{\phi}_j PZ_i^0(\delta_j y)dy\\
&=\begin{cases}
\int_{\R^2}\frac{4\alpha_j^2|y|^{\alpha_j-2}}{(1+|y|^{\alpha_j})^2}\tilde{\phi}_jdy
+O\Big(
\int_{\tilde{\Omega}_j}\Big(|y|^{\alpha_i}
(\frac{\delta_j}{\delta_i})^{\alpha_i} +\delta_i^{\alpha_i}\Big)
\frac{|y|^{\alpha_j-2}}{(1+|y|^{\alpha_j})^2}
\tilde{\phi}_j
\Big)dy,~ &\mbox{ for }i>j,\\
&\\
O\Big(
\int_{\tilde{\Omega}_j}\Big(
\frac{1}{|y|^{\alpha_i}}(\frac{\delta_i}{\delta_j})^{\alpha_i} +\delta_i^{\alpha_i}\Big)
\frac{2\alpha_j^2|y|^{\alpha_j-2}}{(1+|y|^{\alpha_j})^2}
\tilde{\phi}_j\Big)dy,&\mbox{ for }i<j,
\end{cases} \\
& \\
&=\begin{cases}
\frac{2\sigma_j(\lambda)}{|\log \lambda|}+o\bigr(\frac{1}{|\log\lambda|}\bigr),\quad &\mbox{ for }i>j,\\
&\\
o\bigr(\frac{1}{|\log\lambda|}\bigr), &\mbox{ for }i<j,
\end{cases}
\end{split}
\end{equation}
where we used (\ref{projection1}). Next, by Lemma \ref{errorestimate} and (\ref{projection2}),
\begin{equation}\label{term31}
\begin{split}
\rho^+\dfrac{\int_{\Omega}e^WPZ_i^0dx}{\int_{\Omega}e^Wdx}
&=\sum_{j\ even}\int_{\Omega} |x|^{\alpha_j-2}e^{w_j}PZ_i^0dx+
\rho_0\dfrac{\int_{\Omega}e^{z-8k\pi G(x,0)}PZ_i^0dx}{\int_{\Omega}e^{z-8k\pi G(x,0)}dx}+o\Bigr(\frac{1}{|\log\lambda|}\Bigr)\\
&=\int_{\Omega}|x|^{\alpha_i-2}e^{w_i}PZ_i^0 dx+\sum_{j\neq i \ even }\int_{\Omega} |x|^{\alpha_j-2}e^{w_j}PZ_i^0dx+o\Bigr(\frac{1}{|\log\lambda|}\Bigr)\\
&=4\pi \alpha_i+\sum_{j<i\ even}8\pi \alpha_j+o\Bigr(\frac{1}{|\log \lambda|}\Bigr),
\end{split}
\end{equation}
where we replace $\phi$ by $1$ in the estimate of (\ref{eq5}) and (\ref{term2}). Moreover,
\begin{equation}\label{term32}
\begin{split}
\rho^+\dfrac{\int_{\Omega}e^W\phi dx}{\int_{\Omega}e^Wdx}&=\sum_{i \ even}\int_{\Omega} |x|^{\alpha_i-2}e^{w_i}\phi dx
+\rho_0\dfrac{\int_{\Omega}e^{z-8k\pi G(x,0)}\phi dx}{\int_{\Omega}e^{z-8k\pi G(x,0)}dx}+o\Bigr(\frac{1}{|\log\lambda|}\Bigr)
\end{split}
\end{equation}
and again by Lemma \ref{errorestimate} and (\ref{eq5})
\begin{equation}\label{term4}
\begin{split}
\lambda\int_{\Omega}e^{-W}\phi PZ_i^0dx&=\sum_{j\ odd}\int_{\Omega} |x|^{\alpha_j-2}e^{w_j}\phi PZ_i^0dx+o\Bigr(\frac{1}{|\log\lambda|}\Bigr)=\sum_{j<i\ odd}\frac{2\sigma_j(\lambda)}{|\log\lambda|}+o\Bigr(\frac{1}{|\log \lambda|}\Bigr).
\end{split}
\end{equation}
Finally, for the last term,
\begin{equation}\label{term5}
\int_{\Omega}\nabla h\nabla PZ_i^0dx=O(\|h\|\|PZ_i^0\|)=o\Bigr(\frac{1}{|\log\lambda|}\Bigr).
\end{equation}
Combining (\ref{eq4}), (\ref{eq5}), (\ref{term1}), (\ref{term2}), (\ref{term31}), (\ref{term32}), (\ref{term4}) and (\ref{term5}), we deduce that for $i$ even,
\begin{equation}\label{even}
\begin{split}
\frac{4\pi(\alpha_i+\sum_{j<i\ even}2\alpha_j)}{\rho^+}\Big(\sum_{j\ even}\frac{\sigma_j(\lambda)}{|\log\lambda|}+\rho_0\dfrac{\int_{\Omega}e^{z-8k\pi G(x,0)}\phi dx}{\int_{\Omega}e^{z-8k\pi G(x,0)}dx}\Big)
-\frac{1}{|\log\lambda|}(\sigma_i(\lambda)+\sum_{j<i}2\sigma_j(\lambda))
=o\Bigr(\frac{1}{|\log\lambda|}\Bigr).
\end{split}
\end{equation}

Next we consider (\ref{eq4}) for $i$ odd. In this case, again we estimate (\ref{eq4}) term by term. Similarly to the estimate for $i$ even, first by Lemma \ref{errorestimate}, (\ref{projection2}) and (\ref{eq5}), one has
\begin{equation*}
\begin{split}
\rho^+\dfrac{\int_{\Omega}e^W\phi PZ_i^0dx}{\int_{\Omega}e^Wdx}&=\sum_{j\ even}\int_{\Omega}|x|^{\alpha_j-2}e^{w_j}\phi PZ_i^0dx +\rho_0\dfrac{\int_{\Omega}e^{z-8k\pi G(x,0)}\phi PZ_i^0dx}{\int_{\Omega}e^{z-8k\pi G(x,0)}dx}+o\Bigr(\frac{1}{|\log\lambda|}\Bigr)\\
&=\sum_{j<i\ even }\frac{2\sigma_j(\lambda)}{|\log\lambda|}+o\Bigr(\frac{1}{|\log\lambda|}\Bigr),
\end{split}
\end{equation*}
\begin{equation*}
\begin{split}
\rho^+\dfrac{\int_{\Omega}e^W PZ_i^0dx}{\int_{\Omega}e^Wdx}&=\sum_{j\ even}\int_{\Omega} |x|^{\alpha_j-2}e^{w_j} PZ_i^0dx
+\rho_0\dfrac{\int_{\Omega}e^{z-8k\pi G(x,0)} PZ_i^0dx}{\int_{\Omega}e^{z-8k\pi G(x,0)}dx}+o\Bigr(\frac{1}{|\log\lambda|}\Bigr)\\
&=\sum_{j<i\ even }8\pi\alpha_j+o\Bigr(\frac{1}{|\log\lambda|}\Bigr),
\end{split}
\end{equation*}
\begin{equation*}
\begin{split}
\rho^+\dfrac{\int_{\Omega}e^W \phi dx}{\int e^W}&=\sum_{j\ even}\int |x|^{\alpha_j-2}e^{w_j} \phi dx
+\rho_0\dfrac{\int_{\Omega}e^{z-8k\pi G(x,0)}\phi dx}{\int_{\Omega}e^{z-8k\pi G(x,0)}dx}+o\Bigr(\frac{1}{|\log\lambda|}\Bigr)\\
&=\sum_{j\ even}\frac{\sigma_j(\lambda)}{|\log\lambda|}+\rho_0\dfrac{\int_{\Omega}e^{z-8k\pi G(x,0)}\phi dx}{\int_{\Omega}e^{z-8k\pi G(x,0)}dx}+o\Bigr(\frac{1}{|\log\lambda|}\Bigr),
\end{split}
\end{equation*}
and
\begin{equation*}
\begin{split}
\lambda\int_{\Omega}e^{-W}\phi PZ_i^0dx&=\sum_{j\ odd}\int_{\Omega}|x|^{\alpha_j-2}e^{w_j}\phi PZ_i^0dx+o\Bigr(\frac{1}{|\log\lambda|}\Bigr)\\
&=\int_{\Omega}|x|^{\alpha_i-2}e^{w_i}\phi Z_i^0dx+\frac{\sigma_i(\lambda)}{|\log\lambda|}+\sum_{j<i\ odd}\frac{2\sigma_j(\lambda)}{|\log\lambda|}+o\Bigr(\frac{1}{|\log \lambda|}\Bigr).
\end{split}
\end{equation*}
Combining all these terms, one can get that for $i$ odd,
\begin{equation}\label{odd}
\begin{split}
\frac{8\pi\sum_{j<i\ even}2\alpha_j}{\rho^+}\Big(\sum_{j\ even}\frac{\sigma_j(\lambda)}{|\log\lambda|}+\rho_0\dfrac{\int_{\Omega}e^{z-8k\pi G(x,0)}\phi dx}{\int_{\Omega}e^{z-8k\pi G(x,0)}dx}\Big)
-\frac{1}{|\log\lambda|}(\sigma_i(\lambda)+\sum_{j<i}2\sigma_j(\lambda))
=o\Bigr(\frac{1}{|\log\lambda|}\Bigr).
\end{split}
\end{equation}
By considering the difference of (\ref{even}) and (\ref{odd}), one has the following:
\begin{equation}\label{eq7}
\begin{cases}
\frac{4\pi \alpha_{i+1}}{\rho^+}\Big(\sum\limits_{j\ even}\sigma_j(\lambda)+|\log\lambda|\rho_0\dfrac{\int_{\Omega}e^{z-8k\pi G(x,0)}\phi dx}{\int_{\Omega}e^{z-8k\pi G(x,0)}dx}\Big)
-\sigma_{i+1}-\sigma_i=o(1)\quad ~\mbox{ for \ i\ odd},\\
\\
\frac{4\pi \alpha_{i}}{\rho^+}\Big(\sum\limits_{j\ even}\sigma_j(\lambda)+|\log\lambda|\rho_0\dfrac{\int_{\Omega}e^{z-8k\pi G(x,0)}\phi dx}{\int_{\Omega}e^{z-8k\pi G(x,0)}dx}\Big)
-\sigma_{i+1}-\sigma_i=o(1)\quad ~\mbox{ for \ i\ even}.
\end{cases}
\end{equation}
From (\ref{odd}), we first have $\sigma_1(\lambda)=o(1)$. From (\ref{eq7}), we have
\begin{equation*}
\begin{cases}
\sigma_i(\lambda)=o(1)&\mbox{ for \ i\ odd}\\
\\
\sigma_i(\lambda)-\frac{4\pi\alpha_i}{\rho^+}\Big(\sum_{j\ even}\sigma_j(\lambda)+|\log\lambda|\rho_0\dfrac{\int_{\Omega}e^{z-8k\pi G(x,0)}\phi dx}{\int_{\Omega}e^{z-8k\pi G(x,0)}dx}\Big)=o(1)~&\mbox{ for \ i \ even}.
\end{cases}
\end{equation*}

\medskip

\noindent{\bf Step 4. } We claim that $\gamma_i=0$ for $i=1,\cdots,k$. 

\medskip

When $i$ is even, multiplying equation (\ref{linearproblem-1}) by $Pw_i$ and integrating over $\Omega$,
\begin{equation*}
\begin{split}
\int_{\Omega}\nabla \phi \nabla Pw_idx-\rho^+\Big(\dfrac{\int_{\Omega}e^W\phi Pw_idx}{\int_{\Omega}e^Wdx}-\dfrac{\int_{\Omega}e^W\phi dx\int_{\Omega}e^WPw_idx}{(\int_{\Omega}e^Wdx)^2} \Big)
-\lambda\int_{\Omega}e^{-W}\phi Pw_idx=\int_{\Omega}\nabla h\nabla Pw_idx.
\end{split}
\end{equation*}
Now we estimate the above equation term by term. First we have
\begin{equation*}
\int_{\Omega} \nabla \phi \nabla Pw_idx=\int_{\Omega} |x|^{\alpha_i-2}e^{w_i}\phi dx=\int_{\R^2}|y|^{\alpha_i-2}e^{w_i(\delta_i y)}\tilde{\phi}_idy=o(1)
\end{equation*}
by (\ref{limit1-1}) and (\ref{integral1}). To estimate the other terms, by \eqref{expansion-proj} and (\ref{deltai}), we have
\begin{equation}
\label{eq6-1}
\begin{aligned}
&\int_{\Omega}|x|^{\alpha_j-2}e^{w_j}\phi Pw_idx
=\int_{\tilde\Omega_j}\frac{2\alpha_j^2|y|^{\alpha_j-2}}{(1+|y|^{\alpha_j})^2}\tilde{\phi}_jPw_i(\delta_j y)dy\\
&\\
&=\begin{cases}
\int_{\tilde\Omega_j}\frac{2\alpha_j^2|y|^{\alpha_j-2}}{(1+|y|^{\alpha_j})^2}\tilde{\phi}_j(-2\alpha_i\log\delta_i+h_i(0))dy\\
+O\Big(\int_{\tilde\Omega_j}\frac{2\alpha_j^2|y|^{\alpha_j-2}}{(1+|y|^{\alpha_j})^2}\tilde{\phi}_j
(|y|^{\alpha_j}(\frac{\delta_j}{\delta_i})^{\alpha_i}+\delta_j|y|+\delta_i^{\alpha_i} )dy\Big)\quad &\mbox{ for }j<i\\
\\
\int_{\tilde\Omega_i}\frac{2\alpha_i^2|y|^{\alpha_i-2}}{(1+|y|^{\alpha_i})^2}\tilde{\phi}_i
(-2\alpha_i\log\delta_i-2\log(1+|y|^{\alpha_i})+h_i(0))dy\\
+O\Big(\int_{\tilde\Omega_i}\frac{2\alpha_i^2|y|^{\alpha_i-2}}{(1+|y|^{\alpha_i})^2}\tilde{\phi}_i
 (\delta_i|y|+\delta_i^{\alpha_i})dy \Big) &\mbox{ for }j=i\\
 \\
\int_{\tilde\Omega_j}\frac{2\alpha_j^2|y|^{\alpha_j-2}}{(1+|y|^{\alpha_j})^2}\tilde{\phi}_j
(-2\alpha_i\log(\delta_j|y|)+h_i(0))dy\\
+O\Big(\int_{\tilde\Omega_j}\frac{2\alpha_j^2|y|^{\alpha_j-2}}{(1+|y|^{\alpha_j})^2}\tilde{\phi}_j
\Big(\frac{1}{|y|^{\alpha_i}}(\frac{\delta_i}{\delta_j})^{\alpha_i}+\delta_j|y|+\delta_i^{\alpha_i} \Big)dy  \Big) & \mbox{ for }j>i
\end{cases}\\
&\\
&=\begin{cases}
\int_{\tilde\Omega_j}\frac{2\alpha_j^2|y|^{\alpha_j-2}}{(1+|y|^{\alpha_j})^2}\tilde{\phi}_j[-2\alpha_i\log d_i-2(k-i+1)\log\lambda+h_i(0)]dy+o(1) &\mbox{ for } j<i\\
\\
\int_{\tilde\Omega_i}\frac{2\alpha_i^2|y|^{\alpha_i-2}}{(1+|y|^{\alpha_i})^2}\tilde{\phi}_i
[-2\alpha_i\log d_i-2(k-i+1)\log\lambda-2\log(1+|y|^{\alpha_i})+h_i(0)]dy+o(1) &\mbox{ for } j=i\\
\\
\int_{\tilde\Omega_j}\frac{2\alpha_j^2|y|^{\alpha_j-2}}{(1+|y|^{\alpha_j})^2}\tilde{\phi}_j
[-2\alpha_i\log d_j-2(k-j+1)\frac{2i-1}{2j-1}\log\lambda-2\alpha_i\log|y|+h_i(0)]dy+o(1)~&\mbox{ for } j>i.
\end{cases}
\end{aligned}
\end{equation}
Based on \eqref{eq6-1}, by the definition of $\sigma_j(\lambda)$, \eqref{sigma}, \eqref{integral2} and \eqref{integral3}, we get
\begin{equation}
\label{eq6}
\begin{aligned}
&\int_{\Omega}|x|^{\alpha_j-2}e^{w_j}\phi Pw_idx\\
&\\
&=\begin{cases}
-2(k-i+1)\sigma_j(\lambda)+o(1) &\mbox{ for } j<i\\
\\
-2(k-i+1)\sigma_i(\lambda)+\int_{\tilde\Omega_i}\frac{2\alpha_i^2|y|^{\alpha_i-2}}{(1+|y|^{\alpha_i})^2}\tilde{\phi}_i
[-2\log(1+|y|^{\alpha_i})]dy+o(1) &\mbox{ for } j=i\\
\\
-2(k-j+1)\frac{2i-1}{2j-1}\sigma_j(\lambda)+\int_{\tilde\Omega_j}\frac{2\alpha_j^2|y|^{\alpha_j-2}}{(1+|y|^{\alpha_j})^2}\tilde{\phi}_j
[-2\alpha_i\log|y|]dy+o(1)\quad &\mbox{ for } j>i
\end{cases}\\
&\\
&=\begin{cases}
-2(k-i+1)\sigma_j(\lambda)+o(1)  &\mbox{ for } j<i\\
\\
-2(k-i+1)\sigma_i(\lambda)+4\pi\alpha_i\gamma_i+o(1)  &\mbox{ for } j=i\\
\\
-2(k-j+1)\frac{2i-1}{2j-1}\sigma_j(\lambda)+8\pi\alpha_i\gamma_j+o(1) &\mbox{ for } j>i,
\end{cases}
\end{aligned}
\end{equation}
where we used \cite[(4.18)-(4.20) ]{grossi-pistoia}. 

Then by Lemma \ref{errorestimate}, (\ref{eq6}) and (\ref{limit2-1})
\begin{equation*}
\begin{split}
\rho^+\frac{\int_{\Omega}e^W \phi Pw_i dx}{\int_{\Omega}e^Wdx}&=\sum_{j\ even}\int_{\Omega}|x|^{\alpha_j-2}e^{w_j} \phi Pw_idx
+\rho_0\frac{\int_{\Omega}e^{z-8k\pi G(x,0)}\phi Pw_i}{\int_{\Omega}e^{z-8k\pi G(x,0)}dx}+o\Big(\frac{1}{|\log\lambda|}\Big)\\
&=\sum_{j\ even}\int_{\Omega}|x|^{\alpha_j-2}e^{w_j} \phi Pw_idx+o(1)\\
&=4\pi\alpha_i\Bigr(\gamma_i+\sum_{j>i\ even}2\gamma_j\Bigr)-2(k-i+1)\Bigr(\sigma_i(\lambda)+\sum_{j<i\ even}\sigma_j(\lambda)\Bigr)\\
&\quad-\sum_{j>i\ even}2(k-j+1)\frac{2i-1}{2j-1}\sigma_j(\lambda)+o(1).
\end{split}
\end{equation*}
Similarly, by replacing $\phi$ by $1$ in (\ref{eq6}), one can deduce that
\begin{equation*}
\rho^+\dfrac{\int_{\Omega}e^WPw_idx}{\int_{\Omega}e^Wdx}=-8\pi|\log\lambda|\Big(\sum_{j\leq i\ even}(k-i+1)\alpha_j+\sum_{j>i\ even}(k-j+1)\frac{2i-1}{2j-1}\alpha_j\Big)+O(1),
\end{equation*}
and
\begin{equation*}
\rho^+\dfrac{\int_{\Omega}e^W\phi dx}{\int_{\Omega}e^Wdx}=\sum_{i\ even}\frac{\sigma_i(\lambda)}{|\log\lambda|}+\rho_0\dfrac{\int_{\Omega}e^{z-8k\pi G(x,0)}\phi}{\int_{\Omega}e^{z-8k\pi G(x,0)}dx}+o\Big(\frac{1}{|\log\lambda|}\Big).
\end{equation*}
Moreover,
\begin{equation*}
\begin{split}
&\lambda\int_{\Omega}e^{-W}\phi Pw_idx=\sum_{j\ odd}\int_{\Omega} |x|^{\alpha_j-2}e^{w_j}\phi Pw_idx+o(1)\\
&=8\pi \alpha_i\sum_{j>i\ odd}\gamma_j-\sum_{j<i\ odd}2(k-i+1)\sigma_j(\lambda)-\sum_{j>i\ odd}2(k-j+1)\frac{2i-1}{2j-1}\sigma_j(\lambda)+o(1),
\end{split}
\end{equation*}
and
\begin{equation*}
\int_{\Omega}\nabla \phi \nabla Pw_idx
=\int_{\Omega} |x|^{\alpha_i-2}e^{w_i}\phi dx
=\int_{\R^2}|y|^{\alpha_i-2}e^w\tilde{\phi}_idy=o(1)
\end{equation*}
by (\ref{limit1-1}) and (\ref{integral1}). Combining all the above estimates, we get that for $i$ even,
\begin{equation}\label{even1}
\begin{split}
&4\pi\alpha_i(\gamma_i+\sum_{j>i}2\gamma_j)-\sum_{j\leq i}2(k-i+1)\sigma_j-\sum_{j>i}2(k-j+1)\frac{2i-1}{2j-1}\sigma_j\\
&+\frac{8\pi}{\rho^+}\sum_{j\leq i\ even}(k-i+1)\alpha_j\Big(\sum_{l \ even}\sigma_l(\lambda)+|\log\lambda|\rho_0\frac{\int_{\Omega}e^{z-8k\pi G(x,0)}\phi dx}{\int_{\Omega}e^{z-8k\pi G(x,0)}dx} \Big)\\
&+\frac{8\pi}{\rho^+}\sum_{j>i \ even}(k-j+1)\frac{2i-1}{2j-1}\alpha_j
\Big(\sum_{l \ even}\sigma_l(\lambda)
+|\log\lambda|\rho_0\frac{\int_{\Omega}e^{z-8k\pi G(x,0)}\phi dx}{\int_{\Omega}e^{z-8k\pi G(x,0)}dx} \Big)=o(1).
\end{split}
\end{equation}

Next we consider $i$ odd. Similarly to the previous estimates, one has
\begin{equation*}\begin{split}
\rho^+\dfrac{\int_{\Omega}e^W \phi Pw_i dx}{\int_{\Omega}e^Wdx}&=\sum_{j\ even}\int_{\Omega}|x|^{\alpha_j-2}e^{w_j} \phi Pw_idx
+\rho_0\frac{\int_{\Omega}e^{z-8k\pi G(x,0)}\phi Pw_i}{\int_{\Omega}e^{z-8k\pi G(x,0)}dx}+o(\frac{1}{|\log\lambda|})\\
&=8\pi\alpha_i\sum_{j>i\ even}\gamma_j-\sum_{j<i\ even}2(k-i+1)\sigma_j(\lambda)-\sum_{j>i\ even}2(k-j+1)\frac{2i-1}{2j-1}\sigma_j(\lambda)+o(1)
\end{split}
\end{equation*}
\begin{equation*}
\rho^+\dfrac{\int_{\Omega}e^WPw_idx}{\int_{\Omega} e^Wdx}
=-8\pi|\log\lambda|\Big(\sum_{j<i\ even}(k-i+1)\alpha_j
+\sum_{j>i\ even}(k-j+1)\frac{2i-1}{2j-1}\alpha_j\Big)+O(1),
\end{equation*}
and
\begin{equation*}
\begin{split}
&\lambda\int_{\Omega} e^{-W}\phi Pw_idx=\sum_{j\ odd}\int_{\Omega}  |x|^{\alpha_j-2}e^{w_j}\phi Pw_idx+o(1)\\
&=4\pi\alpha_i\gamma_i+8\pi \alpha_i\sum_{j>i\ odd}\gamma_j
-\sum_{j\leq i\ odd}2(k-i+1)\sigma_j(\lambda)-\sum_{j>i\ odd}2(k-j+1)\frac{2i-1}{2j-1}\sigma_j(\lambda)+o(1).
\end{split}
\end{equation*}
So we have for $i $ odd,
\begin{equation}\label{odd1}
\begin{split}
&4\pi\alpha_i(\gamma_i+\sum_{j>i}2\gamma_j)-\sum_{j\leq i}2(k-i+1)\sigma_j-\sum_{j>i}2(k-j+1)\frac{2i-1}{2j-1}\sigma_j\\
&+\frac{8\pi}{\rho^+}\sum_{j< i\ even}(k-i+1)\alpha_j\Big(\sum_{l \ even}\sigma_l(\lambda)+|\log\lambda|\rho_0\frac{\int_{\Omega}e^{z-8k\pi G(x,0)}\phi dx}{\int_{\Omega}e^{z-8k\pi G(x,0)}dx} \Big)\\
&+\frac{8\pi}{\rho^+}\sum_{j>i \ even}(k-j+1)\frac{2i-1}{2j-1}\alpha_j
\Big(\sum_{l \ even}\sigma_l(\lambda)
+|\log\lambda|\rho_0\frac{\int_{\Omega}e^{z-8k\pi G(x,0)}\phi dx}{\int_{\Omega}e^{z-8k\pi G(x,0)}dx} \Big)=o(1).
\end{split}
\end{equation}
By step 3 we know that the terms in (\ref{even1}) and (\ref{odd1}) containig $\sigma_i$ are of order $o(1)$, and thus
\begin{equation*}
4\pi\alpha_i\Bigr(\gamma_i+\sum_{j>i}2\gamma_j\Bigr)=o(1),
\end{equation*}
from which we deduce that $\gamma_i=0$ for $i=1,\cdots,k$.

\medskip

\noindent{\bf Step 5.} Finally, we derive a contradiction.

\medskip

Multiplying equation (\ref{linearproblem-1} ) by $\phi$ and integrating, we get
\begin{equation*}
\begin{split}
\int_{\Omega}|\nabla \phi|^2dx-\lambda\int_{\Omega}e^{-W}\phi^2dx-\rho^+\Big( \dfrac{\int_{\Omega}e^W\phi^2dx}{\int_{\Omega}e^Wdx}-\dfrac{(\int_{\Omega} e^W\phi dx)^2}{(\int_{\Omega} e^Wdx)^2}\Big)
=\int_{\Omega} \nabla h\nabla \phi dx.
\end{split}
\end{equation*}
From Step 1-Step 4 and the assumptions on $\phi$ and $h$, we have that the left hand side of the above equation tends to $1$ while the right hand side is of order $o(1)$. This yields a contradiction.
\qed

\medskip

Once the a priori estimates are carried out, the existence of a solution to the linear problem \eqref{linearproblem-1} follows easily by using the Fredholm alternative, see for example Proposition 5.1 in \cite{dpr1}. 

\medskip

\subsection{Conclusion}
By exploiting the linear theory developed in the previous subsection it is then standard to derive an existence result for the nonlinear problem \eqref{ex-1} based on the contraction mapping, similarly to Proposition \ref{existence}. We skip here the full argument referring to Proposition 5.4 in \cite{dpr1} for full details.
\begin{proposition}\label{existence-1}
For any $\ve>0$ sufficiently small, there exist $\lambda_0>0$ and $C>0$ such that for any $\lambda\in (0, \lambda_0)$, there exists a unique $\phi\in \mathcal{H}_l$ solving
\begin{equation} \label{ex-1}
\Delta(W+\phi)+\rho^+\frac{e^{W+\phi}}{\int_{\Omega}e^{W+\phi}dx}-\lambda e^{-W-\phi}=0\quad\mbox{ in }\quad \Omega
\end{equation}
and
\begin{equation}\label{eq9}
\|\phi\|\leq C\lambda^{\frac{1}{2(2k-1)}-\ve}.
\end{equation}
\end{proposition}

\medskip
\noindent{\bf Proof of Theorem \ref{thm-1}. } By Proposition \ref{existence-1}, $u_\lambda=W_{\lambda}+\phi_\lambda$ is a solution to the original problem (\ref{sinh-gordon}) with $\rho^+_\lambda=\rho^+=4\pi k(k-1)+\rho_0$ and
$\rho^-_{\lambda}=\lambda \int_{\Omega}e^{-u}dx$.
Then by Lemma \ref{errorestimate} and (\ref{eq9})
\begin{equation*}
\begin{split}
\rho^-_\lambda&=\lambda\int_{\Omega}e^{-u}dx=\lambda \int_{\Omega}e^{-W}dx+o(1)
=\sum_{i\ odd}\int_{\Omega}|x|^{\alpha_i-2}e^{w_i}dx+o(1)\\
&=\sum_{i\ odd}4\pi\alpha_i+o(1)=4\pi k(k+1)+o(1).
\end{split}
\end{equation*}
\qed






\begin{thebibliography}{10}

\bibitem{ajy}
W. Ao, A.Jevnikar, and W.Yang.
\newblock On the boundary behavior for the blow up solutions of the sinh-Gordon equation and rank $N$ Toda systems in bounded domains.
\newblock \emph{Int. Math. Res. Not. (IMRN)}, doi:10.1093/imrn/rny263, (2018).

\bibitem{bartolucci-lin}
D. Bartolucci, and C.S. Lin.
\newblock Uniqueness results for mean field equations with singular data.
\newblock {\em Comm. in PDEs},
  34(7):676--702, 2009.

\bibitem{bjmr}	
L. Battaglia, A. Jevnikar, A. Malchiodi, and D. Ruiz. 
\newblock A general existence result for the Toda system on compact surfaces. 
\newblock \emph{Adv. Math.},
 285: 937--979, 2015. 	

\bibitem{bm}
H. Brezis, and F. Merle.
\newblock Uniform estimates and blow-up behavior for solutions of $-\Delta u = V (x)e^u$ in two dimensions.
\newblock \emph{Comm. PDEs} 16: 1223--1253, 1991.

\bibitem{dpkm}
M. Del Pino, M. Kowalczyk, and M. Musso.
\newblock Singular limits in Liouville-type equations.
\newblock \emph{Calc. Var. and PDEs} 24: 47--81, 2005. 

\bibitem{dpr}
T. D'Aprile, A. Pistoia, and D. Ruiz.
\newblock A continuum of solutions for the $SU(3)$ Toda system exhibiting
  partial blow-up.
\newblock {\em Proceedings of the London Mathematical Society},
  111(4):797--830, 2015.

\bibitem{dpr1}
T. D'Aprile, A. Pistoia, and D. Ruiz.
\newblock Asymmetric blow-up for the $SU(3)$ Toda system.
\newblock {\em J. Funct. Anal.}, 271(3):495--531, 2016.

\bibitem{efp}
P. Esposito, P. Figueroa, and A. Pistoia.
\newblock On the mean field equation with variable intensities on pierced domains.  
\newblock \emph{Nonlinear Anal.} 190, 111597, 2020.

\bibitem{grossi-pistoia}
M. Grossi, and A. Pistoia.
\newblock Multiple blow-up phenomena for the sinh-Poisson equation.
\newblock {\em Arch. Rational Mech. Anal.}, 209(1):287--320,
  2013.

\bibitem{jev1}	
A. Jevnikar. 
\newblock An existence result for the mean field equation on compact surfaces in a doubly supercritical regime.
\newblock \emph{Proc. Royal Soc. Edinb. A} 143(5): 1021--1045, 2013.	

\bibitem{jev2}
A. Jevnikar. 
\newblock New existence results for the mean field equation on compact surfaces via degree theory.
\newblock \emph{Rend. Semin. Mat. Univ. Padova} 136: 11--17, 2016. 

\bibitem{jev3}
A. Jevnikar. 
\newblock Multiplicity results for the mean field equation on compact surfaces.
\newblock Adv. Nonlinear Stud. 16(2): 221--229, (2016).

\bibitem{jev4}
A. Jevnikar. 
\newblock Blow-up analysis and existence results in the supercritical case for an asymmetric mean field equation with variable intensities.
\newblock \emph{J. Diff. Eq.} 263: 972--1008, 2017.

\bibitem{jevnikar-wei-yang}
A. Jevnikar, J. Wei, and W. Yang.
\newblock Classification of blow-up limits for the Sinh-Gordon equation.
\newblock  \emph{Differential and Integral Equations} 31(9-10): 657--684, 2018.

\bibitem{jevnikar-wei-yang2}
A. Jevnikar, J. Wei, and W. Yang.
\newblock  On the Topological degree of the Mean field equation with two parameters.
\newblock \emph{Indiana Univ. Math. J.} 67(1): 29--88, 2018.

\bibitem{jevnikar-yang}
A. Jevnikar, and W. Yang.
\newblock  Analytic aspects of the Tzitz\'eica equation: blow-up analysis and existence results.
\newblock \emph{Calc. Var. and PDEs} 56(2): 56:43, 2017.

\bibitem{jost-lin-wang}
J. Jost, C.S. Lin, and G. Wang.
\newblock Analytic aspects of the Toda system: II. Bubbling behavior and
  existence of solutions.
\newblock {\em Comm. Pure Appl. Math.}, 59(4):526--558,
  2006.

\bibitem{jwyz}
J. Jost, G. Wang, D. Ye, and C. Zhou.
\newblock The blow up analysis of solutions of the elliptic sinh-Gordon
  equation.
\newblock {\em Calc. Var. and PDEs},
  31(2):263--276, 2008.

\bibitem{Montgomery}
G. Joyce, and D. Montgomery.
\newblock Negative temperature states for the two-dimensional guiding-centre
  plasma.
\newblock {\em J. Plasma Phys.}, 10(1):107--121, 1973.

\bibitem{li}
Y.Y. Li.
\newblock On a singularly perturbed elliptic equation.
\newblock {\em Adv. Differential Equations}, 2(6):955--980, 1997.



\bibitem{ohtsuka-suzuki}
H. Ohtsuka, and T. Suzuki.
\newblock Mean field equation for the equilibrium turbulence and a related
  functional inequality.
\newblock {\em Adv. Differential Equations}, 11(3):281--304, 2006.

\bibitem{onsager}
L. Onsager.
\newblock Statistical hydrodynamics.
\newblock {\em Il Nuovo Cimento (1943-1954)}, 6(2):279--287, 1949.

\bibitem{pistoia-ricchiardi}
A. Pistoia, and T. Ricciardi.
\newblock Concentrating solutions for a Liouville type equation with variable
  intensities in 2d-turbulence.
\newblock {\em Nonlinearity}, 29(2):271, 2016.

\bibitem{pistoia-ricchiardi2}
A. Pistoia, and T. Ricciardi. 
\newblock Sign-changing tower of bubbles for a sinh-Poisson equation with asymmetric exponents.
\newblock \emph{Discrete Contin. Dyn. Syst.} 37: 5651-–5692, 2017. 

\bibitem{ricciardi-takahashi}
T. Ricciardi, and R. Takahashi.
\newblock Blow-up behavior for a degenerate elliptic
  sinh-Poisson equation with variable intensities.
\newblock {\em Calc. Var. and PDEs},
  55(6):152, 2016.


\bibitem{wente}
H. Wente.
\newblock Counterexample to a conjecture of H. Hopf.
\newblock {\em Pacific J. Math.}, 121(1):193--243, 1986.

\end{thebibliography}
\end{document}